\documentclass[10pt,a4paper]{article}
\usepackage[left=1.75cm,right=1.75cm,top=2cm,bottom=2cm]{geometry} 
\usepackage{amssymb}
\usepackage{amsmath}
\usepackage{amsthm}
\usepackage{xcolor}
\usepackage{graphicx}
\usepackage{epstopdf}
\usepackage{enumerate}
\usepackage[]{algorithm2e}
\usepackage{comment}
\usepackage{psfrag}
\usepackage[breaklinks,bookmarks=false]{hyperref}
\usepackage[textsize=tiny,color=yellow!60!green]{todonotes}
\usepackage{graphicx}
\usepackage{color}
\usepackage{mathbbol}
\usepackage{amssymb}   
\DeclareSymbolFontAlphabet{\amsmathbb}{AMSb}%
\usepackage{float}
\usepackage{array}
\usepackage{multirow}
\usepackage{dsfont}
\usepackage{mathrsfs}
\usepackage[numbers]{natbib}
\usepackage{amsfonts}
\usepackage{amssymb}

\newcommand{\bse}{\begin{subequations}}
\newcommand{\ese}{\end{subequations}}

\newcommand{\im}{\imath}
\newcommand{\jm}{\jmath}

\newcommand{\dt}{\textnormal{d}t}
\newcommand{\ddt}{\dfrac{\textnormal{d}}{\textnormal{d}t}}
\newcommand{\dx}{\textnormal{d}x}
\newcommand{\dtau}{\textnormal{d}\tau}
\newcommand{\pt}{\partial_t}

\newcommand{\bT}{\mathbf{T}}
\newcommand{\bU}{\mathbf{U}}
\newcommand{\bR}{\mathbf{R}}
\newcommand{\bW}{\mathbf{W}}
\newcommand{\bP}{\mathbf{P}}
\newcommand{\bA}{\mathbf{A}}
\newcommand{\bL}{\mathbf{L}}
\newcommand{\bid}{\mathbf{I}}
\newcommand{\bK}{\mathbf{K}}
\newcommand{\bSig}{\boldsymbol{\Sigma}}
\newcommand{\bth}{\boldsymbol{\Theta}}

\newcommand{\bu}{\mathbf{u}}
\newcommand{\by}{\mathbf{y}}
\newcommand{\bv}{\mathbf{v}}
\newcommand{\bff}{\mathbf{f}}
\newcommand{\bw}{\mathbf{w}}
\newcommand{\bz}{\mathbf{z}}
\newcommand{\br}{\mathbf{r}}
\newcommand{\be}{\mathbf{e}}

\newcommand{\etab}{\boldsymbol{\eta}}
\newcommand{\btau}{\boldsymbol{\tau}}
\newcommand{\bep}{\boldsymbol{\varepsilon}}
\newcommand{\bsig}{\boldsymbol{\sigma}}

\DeclareMathOperator{\divr}{div}
\DeclareMathOperator{\tr}{tr}

\DeclareMathOperator{\spann}{span}
\DeclareMathOperator{\ess}{ess}

\newcommand{\Wcal}{\mathcal{W}}
\newcommand{\Ucal}{\mathcal{U}}
\newcommand{\Pcal}{\mathcal{P}}
\newcommand{\Tcal}{\mathcal{T}}
\newcommand{\Rcal}{\mathcal{R}}
\newcommand{\Acal}{\mathcal{A}}
\newcommand{\Scal}{\mathcal{S}}
\newcommand{\Bcal}{\mathcal{B}}
\newcommand{\Ccal}{\mathcal{C}}

\newcommand{\real}{{\amsmathbb R}}
\newcommand{\nat}{{\amsmathbb N}}

\newcommand{\norm}[1]{\left\lVert#1\right\rVert}

\newcommand{\tderiv}[1]{\dfrac{\textnormal{d}#1}{\textnormal{d}t}}

\newtheorem{theorem}{Theorem}[section]

\newtheorem{lemma}[theorem]{Lemma}
\newtheorem{remark}{Remark}[section]
\newtheorem{Hyp}{Hypothesis}
\newtheorem{assum}{Assumption}

\makeatletter
\newcommand{\pushright}[1]{\ifmeasuring@#1\else\omit\hfill$\displaystyle#1$\fi\ignorespaces}
\newcommand{\pushleft}[1]{\ifmeasuring@#1\else\omit$\displaystyle#1$\hfill\fi\ignorespaces}
\makeatother

\begin{document}

\title{Well-posedness  of the fully coupled  quasi-static thermo-poro-elastic equations with nonlinear 
convective transport\thanks{This work forms part of Norwegian Research Council project 250223}}

\author{Mats Kirkes\ae{}ther Brun\footnotemark[2]\and
Elyes Ahmed\footnotemark[2]
\and Florin Adrian Radu\footnotemark[2] 
\and Jan Martin Nordbotten\footnotemark[2]\ \footnotemark[3]
}
\date{\today}
\maketitle

\renewcommand{\thefootnote}{\fnsymbol{footnote}}

\footnotetext[2]{Department of Mathematics, University of Bergen, P. O. Box 7800, N-5020 Bergen, Norway

\href{mailto:mats.brun@uib.no}{mats.brun@uib.no},
\href{mailto:elyes.ahmed@uib.no}{elyes.ahmed@uib.no},
\href{mailto:florin.florin.radu@math.uib.no}{florin.radu@math.uib.no},
\href{mailto:jan.nordbotten@uib.no}{jan.nordbotten@uib.no}
}
\footnotetext[3]{Princeton Environmental Institute, Princeton University, Princeton, N. J., USA.}
\renewcommand{\thefootnote}{\arabic{footnote}}

\numberwithin{equation}{section}

\begin{abstract}
This paper is concerned with the analysis of the quasi-static thermo-poroelastic model. This model is nonlinear and includes thermal effects compared to the classical quasi-static poroelastic model (also known as Biot's model). It consists  of a momentum balance equation, a mass balance equation, and an energy balance  equation, fully coupled and  nonlinear  due to a convective transport term in the energy balance equation. The aim of this article is  to investigate, in the context of mixed formulations,  the existence and  uniqueness of a weak solution to this model problem. The primary variables in these formulations are the fluid pressure, temperature and  elastic displacement as well as the Darcy flux, heat flux and total stress. The well-posedness of a linearized formulation is addressed first through the use of a Galerkin method and suitable \textit{a priori} estimates. This is used next to study the well-posedness of an iterative solution procedure for the full nonlinear problem. A convergence proof for this algorithm is then inferred by a contraction of successive difference functions of the iterates using suitable norms. 
\end{abstract}
\vspace{3mm}



\pagestyle{myheadings} \thispagestyle{plain} \markboth{M. K. Brun}{Analysis of the quasi-static thermo-poro-elastic equations}
\section{Introduction}
The field of \emph{poroelasticity} is concerned with describing the interaction between viscous fluid flow and elastic solid deformation within a porous material, and goes back to the works of K. Terzhagi~\cite{terzaghi1944theoretical} and M. A. Biot~\cite{biot1941general, biot1972theory}. Porous materials are by definition solid materials comprising a great number of interconnected pores, typically at the order of micrometers, where the interconnectivity of the pores is sufficient to allow for fluid flow through the material. For this reason, porous materials are usually modeled at the continuum scale, such that the complex micro-structure needs not be explicitly accounted for in the modeling, but rather implicitly through so-called \emph{effective parameters} such as e.g. \emph{porosity} and \emph{permeability}. Porous materials are primarily associated with objects such as rocks and clays, but biological tissue, foams and paper products also fall within this category. Consequently, the field of \emph{poroelasticity} is of great importance in a range of different engineering disciplines, such as petroleum engineering, agricultural science and biomedicine, among others. A number of comprehensive text books related to the field exists; see e.g.~\cite{coussy2004poromechanics, detournay1995fundamentals, wang2017theory}.

Mathematical modeling of fluid saturated deformable porous media on the continuum scale relies on the theory of linear elasticity, adapted to porous materials by using the so-called \emph{total stress tensor} instead of the Cauchy stress in the momentum balance equation. In particular, the total stress tensor is a linear combination of the Cauchy stress for the empty elastic skeleton and the isotropic stress coming from the fluid, i.e. the pore pressure. Within the quasi-static framework inertial terms are ignored, thus giving a purely elliptic equation for the momentum balance. A second equation of parabolic type accounts for the mass balance as fluid is displaced by the deformation of the solid, and relates change in porosity to volumetric fluid flow, i.e. the \emph{Darcy flux}. This is essentially Biot's poroelastic model for quasi-static deformation (see e.g.~\cite{biot1941general, coussy2004poromechanics}). There is an extensive literature on this model problem and on its numerical approximation. To mention a few, the well-posedness based on the canonical two-field formulation with displacement and pressure as variables was carried out in~\cite{showalter2000diffusion}, while three and four-field formulations have also been analyzed (taking Darcy flux and total stress as independent variables), and can be found in several studies, e.g.~\cite{ahmed:hal-01687026, phillips2008coupling, yi2014convergence}. A key feature of this model, one which greatly facilitates the analysis, is the symmetric coupling between the equations. 

In many important applications, such as geothermal energy extraction, nuclear waste disposal and carbon storage, temperature also plays a vital role and must therefore be included in the modeling. Using the method of formal two-scale expansions (see e.g.~\cite{cioranescu2000introduction, hornung2012homogenization} for a detailed review of this method), a thermo-poroelastic model was derived in~\cite{brun2018thporo}, which accounts for fluid pressure, elastic displacement, and temperature distribution within a fine-grained, fully saturated poroelastic material within the framework of quasi-static deformation. This model is similar to other thermo-poroelastic models which exists in the literature; see e.g. \cite{coussy2004poromechanics, gatmiri1997formulation, lee1997thermal, suvorov2011macroscopic, van2017thermoporoelasticity}, although there are also some notable differences among these works, in particular from the modeling point of view; i.e. allowable flow rates and deformation, choice of coordinate frames etc. (see~\cite{brun2018thporo, van2017thermoporoelasticity} for a comparison of existing thermo-poroelastic models). However, from the point of view of analysis the important factor is the coupling structure between the equations, and the model we analyze exhibits a fully coupled structure.

The aim of the present work is to establish the well-posedness of the nonlinear thermo-poroelastic model previously derived, where we also provide \textit{a priori} energy estimates and regularity of the solutions. We restrict our attention to an isotropic material such that the elastic coefficients are given by the Lam\'e parameters, and the Biot coefficient and thermal stress are given by scalar quantities. Some algebraic constraints on these coefficients must be imposed in order to obtain our results. Although the literature on the analysis of poroelastic models is quite extensive, there is not much literature on the analysis of thermo-poroelastic models. In~\cite{van2017thermoporoelasticity} a corresponding energy functional for the thermo-poroelastic model was derived. This functional was then shown to be monotonically decreasing in time for a small enough characteristic temperature difference.

We undertake our analysis with a future mixed finite-element implementation in mind, and therefore double the number of variables from three to six, and investigate the existence and uniqueness of a weak solution corresponding to this fully coupled six-field model. The primary variables in this model are; fluid pressure, temperature, elastic displacement, Darcy flux, heat flux, and total stress. This makes the problem suitable for combinations of well-known stable finite-elements, such as Raviart-Thomas(-N\'ed\'elec)~\cite{nedelec1986new, raviart1977mixed} and Arnold-Winther~\cite{arnold2007mixed, arnold2002mixed}. From an implementation point of view there are several advantages of a mixed formulation over the canonical three-field formulation; the discretization respects mass and energy conservation, produces continuous normal fluxes regardless of mesh quality, and in general a mixed formulation is advantageous for domain decomposition techniques. We restrict our attention to two spatial dimensions, as this will be the most relevant case for the subsequent work, although the results we present can be extended to higher dimensions in a straightforward manner. In particular, the definition of the isotropic compliance tensor must reflect the choice of spatial dimension. 

The main difficulty we face in the following analysis is the nonlinear coupling between the equations, i.e. the nonlinear convective transport term in the energy balance equation, which takes the form $\nabla T \cdot \bw $, where $\bw$ is the Darcy flux, and $T$ is the temperature distribution. The first part of the paper is concerned with analyzing a linearized version of the model, where we write the convective transport term as $\etab \cdot \bw$, for some given $\etab \in L^\infty$. The analysis of the linearized model retains all the coupling terms of the original problem. Once we have obtained the existence and uniqueness of a weak solution to this problem, we introduce an iterative algorithm where we approximate the convective transport term as $\nabla T^{m-1} \cdot \bw^{m}$, where $m \geq 1$ is the iteration index. Due to the results we obtained for the linearized problem, and by a natural assumption that the temperature gradient admits $L^\infty$-regularity in space, we construct a well defined sequence of iterates as $m \rightarrow \infty$. This we show to converge in adequate norms to the solution of the original nonlinear problem, thus establishing the existence and uniqueness of its weak solution. The convergence proof relies on the Banach Fixed Point Theorem, which we use to obtain local solutions in time. Here, the time interval is supposed to be small to ensure a contraction of the successive difference functions of the iterates. Then, using piecewise continuation in time, we extend these local solutions to global solutions for any finite final time. The idea is that such an iterative scheme can also be applied numerically to a discretized formulation, and in this sense our analysis sets the stage for subsequent numerical experiments. We mention also some of the literature on iterative schemes in poroelasticity;  in~\cite{bause2017space, both2017robust, lee2016robust, mikelic2013convergence} there can be found several iterative procedures for solving Biot's equations, and in~\cite{list2016study, pop2004mixed, radu2015robust} iterative methods for solving Richards' equation were analyzed.

\textbf{We summarize the main contribution of the article as follows}: under a natural hypothesis on the regularity of the convective term, we give a proof of existence and uniqueness of a weak solution to the fully coupled six-field thermo-poroelastic problem within the quasi-static framework. 

The article is organized as follows: Section \ref{sec2} recalls the physical model and the assumptions on the data, introduces the relevant functions spaces and some preliminary results and introduces the mixed weak formulations. In section \ref{sec3}, we define a linear version of the original mixed variational problem, and proceed to analyze this in the following way; we construct approximate solutions using a Galerkin method, the existence of which is established by the theory of DAEs (Differential Algebraic Equations). Suitable \textit{a priori} estimates are then derived which enables us to pass to the limit, thanks to the weak compactness of the spaces. Section \ref{sec4}, is then devoted to analyze a linearization solution procedure for the original nonlinear 
problem and to establish the convergence of the algorithm in suitable norms. In Appendix \ref{appendix} we propose an alternative to the hypothesis on the temperature gradient, i.e. we show how the required regularity can be obtained by sufficient regularity of the data. For easy reference of the notation used in this article we provide some tables in Appendix \ref{appB}.
\section{Presentation of the problem}\label{sec2}
Let $\Omega \subset \real^d$, for $d \in \{2,3\}$, be an open and bounded domain, where we denote the boundary by $\Gamma := \partial \Omega$, which is assumed to be Lipschitz continuous. Let a time interval $J = (0,T_f)$ be given with  $T_f > 0$ and define  $Q := \Omega \times (0,T_f]$ to be the space-time domain.  
The thermo-poroelastic model problem we consider,  as it is exposed in~\cite{brun2018thporo},  is as  follows: given a heat source $h$, a body force $\bff$, and a mass source $g$, find $(T, \bu, p)$ such that
\bse\label{thporo}
	\begin{alignat}{3}
	\partial_t(a_0T - b_0p + \beta \nabla \cdot \bu) - \nabla T \cdot (\bK \nabla p)
		- \nabla \cdot(\bth \nabla T) &=h,\quad&&\textnormal{in } Q,\\
		-(\lambda + \mu)\nabla(\nabla \cdot \bu) - \mu \nabla^2 \bu + \alpha \nabla p + \beta \nabla T &= \bff,\quad&&\textnormal{in }Q,\\
		\partial_t(c_0p - b_0 T + \alpha \nabla \cdot \bu) - \nabla \cdot(\bK \nabla p) &= g,\quad&&\textnormal{in } Q,
	\end{alignat}
where $a_0$ is the effective thermal capacity, $b_0$ is the thermal dilation coefficient, $\beta$ is the thermal stress coefficient, $\bK = (K_{ij})_{i,j=1}^d$ is the permeability divided by fluid viscosity, $\bth = (\Theta_{ij})_{i,j=1}^d$ is the effective thermal conductivity, $\mu$ and $\lambda$ are the Lam\'e parameters, $\alpha$ is the Biot-Willis constant and $c_0$ is the specific storage coefficient. The primary variables are the temperature distribution $T$, displacement $\bu$ and fluid pressure $p$. To close the system, we prescribe homogeneous Dirichlet conditions on the boundary, i.e, 
\begin{equation}\label{BCs}
T = 0, \quad \bu = 0, \quad \textnormal{and } \quad p = 0, \quad \textnormal{on } \Gamma \times J,
\end{equation}
and we assume  the following initial conditions 
\begin{equation}\label{ICs}
		T(\cdot,0) = T_0, \quad \bu(\cdot,0) = \bu_0, \quad \textnormal{and } \quad p(\cdot,0) = p_0, \quad \textnormal{in } \Omega \times \{0\},
\end{equation}
for some known functions $T_0$, $\bu_0$ and $p_0$. 
\ese
In practice, we may use nonhomogeneous Dirichlet and
Neumann boundary conditions for which the analysis remains valid.  Note also 
that if $\beta = b_0 = 0$, the above system decouples from the energy equation, and 
the well-known quasi-static Biot equations are recovered (see e.g.~\cite{ahmed:hal-01687026} where both the two- and four-field formulations are presented). 
\subsection{Preliminaries}\label{prelim}
First, we define the spaces that will be used throughout this article, we refer to e.g.~\cite{evans1998partial, yosida1995functional} for more details. For $1\leq p < \infty$ let $L^p(\Omega) = \{ u: \Omega \rightarrow \real : \int_\Omega |u|^p \dx < \infty\}$, with the associated norm $\norm{\cdot}_p$. In particular, $L^2(\Omega)$ is the Hilbert space of square integrable functions defined on $\Omega$, endowed with the inner product $(\cdot, \cdot)$, and the norm $\norm{\cdot} := \norm{\cdot}_2$. For $p=\infty$, $L^\infty(\Omega)$ is the space of uniformly bounded measurable functions defined on $\Omega$ , i.e. $L^\infty(\Omega) = \{ u:\Omega \rightarrow \real : \ess \sup_{x\in \Omega} |u| \leq \infty \}$, endowed with the norm $\norm{u}_{\infty} = \inf \{ C: |u| \leq C \textnormal{ a.e. on } \Omega\}$. We denote by $W^{k,p}(\Omega)$ the Sobolev space of functions in $L^p(\Omega)$, admitting weak derivatives up to order $k$ in the same space. In particular, we denote by $H^1(\Omega) := W^{1,2}(\Omega) = \{ u \in L^2(\Omega) : \nabla u \in (L^2(\Omega))^d\}$, and designate by $H^1_0(\Omega)$ its zero-trace subspace.  Let $H(\divr,\Omega) = \{ \bv \in (L^2(\Omega))^d : \nabla \cdot \bv \in L^2(\Omega) \}$ be the space of vector valued functions, where each component belongs to $L^2(\Omega)$, along with the weak divergence. We endow this space with the norm $\norm{\bv}^2_{H(\divr; \Omega)} := \norm{\bv}^2 + \norm{\nabla \cdot \bv}^2$. Let $H_s(\divr,\Omega) = \{ \btau \in (L^2(\Omega))^{d \times d} : \nabla \cdot \btau \in (L^2(\Omega))^d, \btau_{ij} = \btau_{ji} \textnormal{ for } 1 \leq i,j \leq d\}$ be the space of symmetric tensor valued functions defined on $\Omega$, where each component belongs to $L^2(\Omega)$, and admitting a weak divergence in $(L^2(\Omega))^d$. We denote by $C^1(\Omega)$ the space of continuous functions defined on $\Omega$, admitting continuous partial derivatives. Finally, let $X$ be a Banach space and let $L^p(J;X)$ be the Bochner space of functions in $L^p$ defined on $J$ with values in $X$. Let $\norm{\cdot}_X$ be a norm on $X$, then for $u \in L^p(J;X)$, $p < \infty$, we have $\norm{u}^p_{L^p(J;X)} := \int_0^{T_f} \norm{u(t)}^p_X\dt$. In particular, we will make use of the spaces $H^1(J;L^2(\Omega)) = \{ u(t) : \Omega \rightarrow \real : \int_0^{T_f} (\norm{u(t)}^2 + \norm{\partial_t u(t)}^2) dt < \infty \}$ and $L^\infty(J;L^2(\Omega)) = \{ u(t) : \Omega \rightarrow \real: \ess \sup_{ t \in J} \norm{u(t)} < \infty \}$. Note that if $\bu(t) \in (L^2(\Omega))^d$ is square integrable in time, we shall still write $\bu \in L^2(J;L^2(\Omega))$, but this should not cause any confusion as we will always utilize bold fonts for vector (or tensor) valued functions. \\

The following classical results will be used in this paper.
\newtheorem{ES}[theorem]{Eberlein-\v Smulian~\cite{yosida1995functional}}
\begin{ES}\label{ES}
Let $X$ be a reflexive Banach space, and $\{x_n\}_{n\geq 1}$ a bounded sequence in X. Then there exists a subsequence $\{x_{n_k} \}_{k\geq 1} \subset \{x_n\}_{n\geq1}$ and $x \in X$ such that $x_{n_k} \rightharpoonup x$ in X, as $k \rightarrow \infty$.
\end{ES}

\newtheorem{CM}[theorem]{Banach fixed point~\cite{cheney2013analysis}}
\begin{CM}\label{CM}
Let $X$ be a Banach space and $U \subseteq X$ a closed subset. If $T: U \rightarrow U$ is a contraction map, i.e. there exist $0 < C < 1$ such that $\norm{T(x) - T(y)} \leq C\norm{x-y}$, $\forall x,y \in U$, then $T$ admits a unique fixed point $x^* \in U$. Moreover, the sequence $\{ x_k \}_{k\geq0}$, where $x_0 \in U$ is arbitrary and $x_{k+1} := T(x_k)$, converges to $x^*$.
\end{CM}


\newtheorem{hdiv}[theorem]{Thomas' Lemma~\cite{thomas1977analyse}}
\begin{hdiv}\label{thomas}
	If $u \in L^2(\Omega)$, then there exists $\bu \in H(\divr,\Omega)$, such that a.e. $\nabla \cdot \bu = u$, and $\norm{\bu} \leq C\norm{u}$, for some constant $C > 0$ depending only on the domain and the spatial dimension.
\end{hdiv}

Additionally we shall frequently apply the following classical inequalities. 

\newtheorem{CS}[theorem]{Cauchy - Schwarz' inequality (C-S)}
\begin{CS}\label{CS}
Given functions $u,v \in L^2(\Omega)$, there holds $(u,v) \leq \norm{u} \norm{v}$.
\end{CS}

\newtheorem{young}[theorem]{Young's inequality}
\begin{young}\label{young}
For any $a,b \in \real$, there holds $|ab| \leq \dfrac{\epsilon}{2} a^2 + \dfrac{1}{2\epsilon} b^2$, for any $\epsilon > 0$.
\end{young}

\newtheorem{gr}[theorem]{Gr\"onwall's Lemma}
\begin{gr}\label{gr}
For any continuous function $u$, and integrable and non-decreasing $A$, defined on an interval $I = [a,b]$, such that there holds 
\begin{equation*}
	u(t) \leq A(t) + B \int_a^tu(\tau)\dtau, \quad \forall t \in I,
\end{equation*}
for any constant $B$, then
\begin{equation*}
	u(t) \leq A(t)e^{B(t-a)}, \quad \forall t \in I.
\end{equation*}
\end{gr}
\subsection{Assumptions on the data}\label{assumptions}
Before transcribing the mixed variational formulation of the problem~\eqref{thporo}, we make precise the assumptions on
the data (further generalizations are possible, bringing more technicalities):
\begin{assum}[Data]\label{assumpdata}
\begin{enumerate}[{A}.1]
\item The source terms are such that $\ g, h \in L^2(J;L^2(\Omega)),\textnormal{ and }\bff \in H^1(J;L^2(\Omega)) $.
\item The initial conditions are such that  $p_0, T_0 \in H^1_0(\Omega), \textnormal{ and } \bu_0 \in (L^2(\Omega))^d$.
\item The permeability and heat conductivity tensors are such that $\bK, \bth  \in (L^\infty(\Omega))^{d\times d}$. Furthermore, we assume there exists $k_M, k_m > 0$ such that for a.e. $x \in \Omega$ there holds
\begin{equation*} 
k_m |\zeta|^2 \leq \zeta^T \bK^{-1}(x) \zeta \textnormal{ and } |\bK^{-1}(x)\zeta| \leq k_M |\zeta|, \ \forall \zeta \in \real^d \setminus \{0\},
\end{equation*} 
and there exists $\theta_M, \theta_m > 0$ such that for a.e. $x \in \Omega$ there holds 
\begin{equation*}
\theta_m |\zeta|^2 \leq \zeta^T \bth^{-1}(x) \zeta \textnormal{ and } |\bth^{-1}(x)\zeta| \leq \theta_M |\zeta|, \ \forall \zeta \in \real^d \setminus \{0\}.
\end{equation*}
\item The constants $c_0, b_0, a_0, \alpha, \beta, \mu$, and $\lambda$, are strictly positive.
\end{enumerate}
\end{assum}

\subsection{Mixed variational formulation}\label{var}
We now give the mixed variational formulation of the problem~\eqref{thporo}, for which we need to introduce  the total stress tensor;
$\bsig(\bu, p, T) := 2 \mu \bep(\bu) + \lambda \nabla \cdot \bu \bid - \alpha p \bid - \beta T \bid$, where $\bid$   is the identity tensor
and  $\bep(\bu)$ is  the linearized strain tensor 
given by $\bep(\bu):=(\nabla \bu+\nabla^{\textnormal{T}}\bu)/2$, the Darcy flux  $\bw := -\bK \nabla p$, 
and the heat flux  $\br := -\bth \nabla T$.  For simplicity, we now restrict our attention to the case $d=2$, in which case the fourth order compliance tensor, $\Acal$, is given by
\begin{equation}\label{compliance}
	\Acal \btau := \frac{1}{2\mu} \left( \tau - \frac{\lambda}{2(\mu+\lambda)} \tr(\btau) \bid \right), 
	\quad \btau \in \real^{d\times d},
\end{equation}
as seen in~\cite{yi2014convergence} (see also~\cite{lee2016robust} for the general formula). Note that $\Acal$ is bounded and symmetric positive definite uniformly with respect to $x \in \Omega$, and defines an $L^2$-equivalent norm, i.e.
\begin{equation}\label{complnorm}
	\frac{1}{2(\mu + \lambda)} \norm{\btau}^2 \leq \norm{\btau}_{\Acal}^2 \leq \frac{1}{2\mu} \norm{\btau}^2,\qquad\forall \btau\in \left(L^2(\Omega)\right)^{d\times d},
\end{equation}
where $\norm{\btau}^2_{\Acal} = \int_\Omega \Acal \btau : \btau \dx$. Applying $\Acal$ to the total stress tensor, it is inferred that
\begin{equation}
	\Acal \bsig = \bep(\bu) - \frac{1}{2(\mu+\lambda)} (\alpha p + \beta T)\bid,
\end{equation}
and by taking the trace on both sides, we get the following relationship
\begin{equation}\label{trsig}
\nabla \cdot \bu = \frac{1}{2(\mu + \lambda)} \tr(\bsig) + \frac{1}{\mu + \lambda}(\alpha p+\beta T).
\end{equation}
We also introduce the following notation
\begin{equation}
c_r := \frac{\alpha^2}{\mu + \lambda}, \qquad b_r := b_0 - \frac{\alpha \beta}{\mu + \lambda}, \qquad a_r := \frac{\beta^2}{\mu + \lambda}.
\end{equation}
The  above definitions yields an equivalent mixed form to~\eqref{thporo}:
\bse\label{eqns}
 	\begin{alignat}{2}
 		\partial_t ( a_0 T - b_0 p + \beta \nabla \cdot \bu) + \nabla T \cdot \bw + \nabla \cdot \br &= h,\quad
 		&&\textnormal{in } Q, \\
 		\bth^{-1} \br + \nabla T &= 0,\quad 
 		&&\textnormal{in } Q,\\
 		\partial_t (c_0 p - b_0 T + \alpha \nabla \cdot \bu) + \nabla \cdot \bw &= g,\quad
 		&&\textnormal{in } Q, \\
 		\bK^{-1} \bw + \nabla p &= 0,\quad
 		&&\textnormal{in } Q,\\
 		\Acal \bsig - \bep(\bu) + \frac{c_r}{2\alpha}\bid p + \frac{a_r}{2\beta}\bid T&= 0,\quad
 		&&\textnormal{in } Q,  \\
 		-\nabla \cdot \bsig &= \bff,\quad &&\textnormal{in } Q. 		
 	\end{alignat}
\ese
We now set $$\Tcal := L^2(\Omega),\ \ \Rcal := H(\divr,\Omega),\ \ \Pcal := L^2(\Omega),\ \ \Wcal := H(\divr,\Omega), \ \ \Scal := H_s(\divr,\Omega),\ \ \Ucal := (L^2(\Omega))^d.$$
The following mixed variational formulation of the problem \eqref{thporo} can  be obtained by multiplying by adequate 
test functions and then integrating by parts: find $(T(t), \br(t), p(t), \bw(t), \bsig(t), \bu(t)) \in \Tcal \times \Rcal \times \Pcal \times \Wcal \times \Scal \times \Ucal$, such that a.e. for $t \in J$ there holds
\bse\label{nonlinvar}
	\begin{alignat}{2}
		(a_0 + a_r)(\partial_t T, S) - b_r (\partial_t p, S) + \frac{a_r}{2\beta}(\partial_t \bsig, S\bid) 
		- (\bth^{-1} \br \cdot \bw, S) + (\nabla \cdot \br,S) &= (h,S), \quad 
		&&\forall S \in \Tcal,\label{nonlinenergy}\\
		(\bth^{-1} \br, \by) - (T, \nabla \cdot \by) &= 0, \quad
		&&\forall \by \in \Rcal, \\
		(c_0 + c_r)(\partial_t p, q) -b_r(\partial_t T, q) + \frac{c_r}{2\alpha} (\partial_t \bsig, q\bid) 
		+ (\nabla \cdot \bw,q) &= (g,q), \quad
		&&\forall q \in \Pcal, \\
		(\bK^{-1} \bw, \bz) - (p,\nabla \cdot \bz) &= 0, \quad 
		&&\forall \bz \in \Wcal, \\
		(\Acal \bsig, \btau) + (\bu, \nabla \cdot \btau) + \frac{c_r}{2\alpha}(\bid p, \btau) 
		+ \frac{a_r}{2\beta}(\bid T,\btau) &= 0, \quad
		&&\forall \btau \in \Scal,\\
		-(\nabla \cdot \bsig, \bv) &= (\bff,\bv), \quad
		&&\forall \bv \in \Ucal, 
	\end{alignat}
and such that the initial conditions \eqref{ICs} holds true in the weak sense, i.e.
\begin{equation}\label{weakICs}
(p(0),q) = (p_0,q) \quad \forall q \in \Pcal, \quad ( \bu(0),\bv) = ( \bu_{0},\bv) \quad \forall \bv \in \Ucal, \quad \textnormal{ and } \quad
(T(0),S) = (T_0,S) \quad \forall S \in \Tcal.
\end{equation}
\ese
\begin{remark}
Note that a different variational formulation of the problem \eqref{eqns} is possible, using a weakly symmetric space for the stress tensor. This formulation will then involve a new variable acting as a Lagrange multiplier which is enforcing the symmetry of the stress (see e.g.~\cite{arnold2007mixed, baerland2017weakly, lee2016robust}). For simplicity of presentation we shall keep the formulation \eqref{nonlinvar} throughout. The analysis presented next can nevertheless also be extended to the previously mentioned formulation using the same techniques, as done in~\cite{ahmed:hal-01687026} for the four-field Biot equations.
\end{remark}

\begin{remark}
The nonlinear coupling in the above problem makes the analysis difficult. The next section is therefore devoted to analyzing a linearized problem, the results from which will be helpful when analyzing the full nonlinear problem in the last section. We mention also that other nonlinearities can be added, e.g. nonlinear compressibility or nonlinear Lam\'e parameters. 
\end{remark}
\section{Analysis of the linear problem}\label{sec3}
In this section we introduce a linear version of the problem \eqref{nonlinvar}. Precisely, we replace the convective 
transport term $-(\bth^{-1} \br \cdot \bw, S)$ in the energy balance equation \eqref{nonlinenergy}, 
by $-(\etab \cdot \bw,S)$, for some given $\etab \in L^\infty(\Omega)$. 
We denote by $\gamma := \norm{\etab}_{\infty}$.  We  
introduce the  resulting linear problem which reads: find $(T(t), \br(t), p(t), \bw(t), \bsig(t), \bu(t)) \in \Tcal \times \Rcal \times \Pcal \times \Wcal \times \Scal \times \Ucal$, such that for a.e. $t \in J$ there holds
\begin{subequations}\label{linearized}
	\begin{alignat}{2}
		(a_0 + a_r)(\partial_t T, S) - b_r (\partial_t p, S) + \frac{a_r}{2\beta}(\partial_t \bsig, S\bid) 
		- (\etab \cdot \bw, S) + (\nabla \cdot \br,S) &= (h,S), \quad
		&&\forall S \in \Tcal,\\
		(\bth^{-1} \br, \by) - (T, \nabla \cdot \by) &= 0, \quad 
		&&\forall \by \in \Rcal,\\
		(c_0 + c_r)(\partial_t p, q) -b_r(\partial_t T, q) + \frac{c_r}{2\alpha} (\partial_t \bsig, q\bid) 
		+ (\nabla \cdot \bw,q) &= (g,q), \quad
		&&\forall q \in \Pcal, \\
		(\bK^{-1} \bw, \bz) - (p,\nabla \cdot \bz) &= 0, \quad 
		&&\forall \bz \in \Wcal, \\
		(\Acal \bsig, \btau) + (\bu, \nabla \cdot \btau) + \frac{c_r}{2\alpha}(\bid p, \btau) 
		+ \frac{a_r}{2\beta}(\bid T,\btau) &= 0, \quad
		&&\forall \btau \in \Scal,\\
		-(\nabla \cdot \bsig, \bv) &= (\bff,\bv), \quad
		&&\forall \bv \in \Ucal, 
	\end{alignat}
\end{subequations}
and such that initial conditions \eqref{weakICs} holds true. The remaining part of this section is devoted to proving the well-posedness of this system. 
In what follows, we assume the following hypothesis on the effective thermal capacity $a_0$,  
the thermal dilation coefficient $b_0$, the specific storage coefficient $c_0$ and the Lam\'e parameters $\mu, \lambda$;  
\begin{equation} \label{constraintcoeff}
b_0 - \dfrac{\alpha \beta}{\mu + \lambda} > 0,\qquad \qquad c_0 - \dfrac{c_r}{2} - b_0 - \dfrac{1}{6(\mu + \lambda)} > 0, \qquad
a_0 - \dfrac{a_r}{2} - b_0 - \dfrac{1}{6(\mu + \lambda)} > 0.
\end{equation}

These constraints are typically needed in order to ensure a gradient flow structure. Similar constraints were used to analyze the Biot equations in mixed form in~\cite{ahmed:hal-01687026}. We also refer the reader to~\cite{lee2017parameter} for a more detailed discussion about the scaling of Biot's (isothermal) equations. However, compared to these works, our constraints involve also the thermal coefficients. We omit any further discussion on the justification for these constraints, other than they are necessary to prove the results we present. 
The well-posedness of problem~\eqref{linearized} is then given in the following result.
\begin{theorem}[Well-posedness of the linear problem]\label{linWP}
Under Assumption~\ref{assumpdata}, the problem \eqref{linearized}, \eqref{weakICs} has a unique solution
\begin{subequations}
	\begin{align}
		&(T,\br) \in  H^1(J;L^2(\Omega))\times \left( L^2(J;H(\divr;\Omega)) \cap L^\infty(J;L^2(\Omega)) \right),\\
		&( p,\bw) \in  H^1(J;L^2(\Omega))\times \left( L^2(J;H(\divr;\Omega)) \cap L^\infty(J;L^2(\Omega)) \right), \\
		&(\bu,\bsig) \in   H^1(J;L^2(\Omega))\times \left(L^2(J;H_s(\divr;\Omega)) \cap H^1(J;L^2(\Omega))\right).
	\end{align}
\end{subequations}
Moreover, if $g,h \in H^1(J;L^2(\Omega))$, $\bff \in H^2(J;L^2(\Omega))$ then
\begin{subequations}
	\begin{align}
		&(T,\br) \in W^{1, \infty}(J;L^2(\Omega))\times \left( L^\infty(J;H(\divr;\Omega)) \cap H^1(J;L^2(\Omega)) \right),\\
		&(p,\bw) \in W^{1, \infty}(J;L^2(\Omega))\times \left( L^\infty(J;H(\divr;\Omega)) \cap H^1(J;L^2(\Omega)) \right) , \\
		&(\bu,\bsig) \in W^{1, \infty}(J;L^2(\Omega))\times \left(L^{\infty}(J;H_s(\divr;\Omega)) \cap W^{1, \infty}(J;L^2(\Omega))\right).
	\end{align}
\end{subequations}
\end{theorem}
The proof will follow from a series of partial results to be done in the sequel. The analysis uses a Galerkin's method together with the theory of differential algebraic equations (DAEs), as well as 
weak compactness arguments~(cf. \cite{ahmed:hal-01687026, yi2014convergence, phillips2008coupling, evans1998partial}).
\subsection{Construction of approximate solutions}\label{existence}
First, we need to introduce the following finite dimensional subspaces. Let $(i,j,k,l,m,n) \in \nat^6$ be fixed and strictly positive, and let $\Tcal_i := \spann \{ S_{\ell} \in \Tcal: \ell = 1, \cdots, i \}$, $\Rcal_j := \spann \{ \by_\ell \in \Rcal: \ell = 1, \cdots, j \}$, $\Pcal_k := \spann \{ q_\ell \in \Pcal: \ell = 1, \cdots, k \}$, $\Wcal_l := \spann \{ \bz_\ell \in \Wcal: \ell = 1, \cdots, l \}$, $\Scal_m := \spann \{ \btau_\ell \in \Scal: \ell = 1, \cdots, m \}$ and $\Ucal_n := \spann \{ \bv_\ell \in \Ucal : \ell = 1, \cdots, n \}$, where the functions $S_\ell, \by_\ell, q_\ell, \bz_\ell, \btau_\ell$ and $\bv_\ell$, for $\ell \in \nat$, constitute Hilbert bases for the spaces $\Tcal, \Rcal, \Pcal, \Wcal, \Scal$ and $\Ucal$, respectively. Let now $(T_i, \br_j, p_k, \bw_l, \bsig_m, \bu_n) : [0,T_f]^6 \rightarrow \Tcal_i \times \Rcal_j \times \Pcal_k \times \Wcal_l \times \Scal_m \times \Ucal_n$ be the solution to the following problem:
\bse\label{fdvar}
	\begin{alignat}{2}
		\nonumber(a_0 + a_r)(\partial_t T_i, S_\ell) - b_r(\partial_t p_k, S_\ell) + \frac{a_r}{2\beta}(\partial_t \bsig_m, S_\ell \bid) \qquad \qquad \qquad& \\
		- (\etab \cdot \bw_l, S_\ell) + (\nabla \cdot \br_j, S_\ell) &= (h,S_\ell),\quad
		&&\ell = 1,\cdots,i,\label{fd5}\\[2ex]
		(\bth^{-1} \br_j, \by_\ell) - (T_i, \nabla \cdot \by_\ell) &= 0, \quad
		&&\ell = 1,\cdots,j \label{fd6},\\
		(c_0 + c_r)(\partial_t p_k, q_\ell) - b_r(\partial_t T_i, q_\ell) 
		+ \frac{c_r}{2\alpha} (\partial_t \bsig_m, q_\ell \bid) + (\nabla \cdot \bw_l, q_\ell) 
		&= (g,q_\ell),\quad
		&&\ell = 1,\cdots,k, \label{fd3}\\[2ex]
		(\bK^{-1} \bw_l, \bz_\ell) - (p_k,\nabla \cdot \bz_\ell) &= 0, \quad
		&&\ell = 1,\cdots,l, \label{fd4}\\[2ex]
		(\Acal \bsig_m, \btau_\ell) + (\bu_n, \nabla \cdot \btau_\ell) 
		+ \frac{c_r}{2\alpha}(\bid p_k, \btau_\ell) + \frac{a_r}{2\beta}(\bid T_i,\btau_\ell) &= 0, \quad
		&&\ell = 1,\cdots,m,\label{fd1}\\[2ex]
		-(\nabla \cdot \bsig_m, \bv_\ell) &= (\bff,\bv_\ell), \quad
		&&\ell = 1,\cdots,n. \label{fd2}
	\end{alignat}

We introduce the coefficient vectors of the solutions: let $\bT_i(t) := [T_1(t),\cdots,T_i(t)]^T$ where $T_i(x,t) = \sum_{\ell=1}^i T_\ell(t) S_\ell$, $\bR_j(t) := [r_1(t),\cdots,r_j(t)]^T$ where $\br_j(x,t) = \sum_{\ell=1}^j r_\ell(t) \by_\ell$, $\bP_k(t) := [p_1(t), \cdots, p_k(t)]^T$ where $p_k(x,t) = \sum_{\ell=1}^k p_\ell(t) q_\ell$, $\bW_l(t) := [w_1(t),\cdots,w_l(t)]^T$ where $\bw_l(x,t) = \sum_{\ell=1}^l w_\ell(t) \bz_\ell$, $\bSig_m(t) := [\sigma_1(t), \cdots, \sigma_m(t)]^T$ where $\bsig_m(x,t) = \sum_{\ell=1}^m \sigma_\ell(t) \btau_\ell$ and $\bU_n(t) := [u_1(t), \cdots u_n(t)]^T$ where $\bu_n(x,t) = \sum_{\ell=1}^n u_\ell(t) \bv_\ell$. \\

Thus, we impose the initial conditions by
\begin{equation}\label{fdICs}
	T_\ell(0) = (T_0, S_\ell), \ 1\leq \ell \leq i, \quad u_\ell(0) = (\bu_0, \bv_\ell), \ 1\leq \ell \leq n, 
	\quad p_\ell(0) = (p_0, q_\ell), \ 1\leq \ell \leq k.
\end{equation}
\ese
We also define the following linear operators: $(\bA_{\bsig \bsig})_{\im \jm} := (\Acal \btau_\im,\btau_\jm)$, for $1 \leq \im, \jm \leq m$, $(\bA_{pp})_{\im \jm} := (c_0 + c_r)(q_\im,q_\jm)$, for $1 \leq \im,\jm \leq k$, $(\bA_{TT})_{\im \jm} := (a_0 + a_r)(S_\im,S_\jm)$, for $1\leq \im,\jm \leq i$, $(\bA_{\bw \bw})_{\im \jm} := (\bK^{-1} \bz_\im,\bz_\jm)$, for $1\leq \im,\jm \leq l$, $(\bA_{\br \br})_{\im \jm} := (\bth^{-1} \by_\im, \by_\jm)$, for $1 \leq \im,\jm \leq j$, $(\bA_{\bu \bsig})_{\im \jm} := (\bv_\im, \nabla \cdot \btau_\jm)$, for $1\leq \im \leq n, 1\leq \jm \leq m$, $(\bA_{p \bsig})_{\im \jm} := \dfrac{c_r}{2\alpha}(\bid q_\im, \btau_\jm)$, for $1\leq \im \leq k, 1\leq \jm \leq m$, $(\bA_{T\bsig})_{\im \jm} := \dfrac{a_r}{2\beta}(\bid S_\im, \btau_\jm)$, for $1\leq \im \leq i, 1\leq \jm \leq m$, $(\bA_{Tp})_{\im \jm} := -b_r(S_\im,q_\jm)$, for $1\leq \im \leq i, 1\leq \jm \leq k$, $(\bA_{\bw p})_{\im \jm} := (\nabla \cdot \bz_\im, q_\jm)$, $1\leq \im \leq l, 1\leq \jm \leq k$, $(\bA_{\br T})_{\im \jm} := (\nabla \cdot \by_\im, S_\jm)$, for $1\leq \im \leq l, 1\leq \jm \leq i$, and $(\bA_{\bw T})_{\im \jm} := (\etab \cdot \bz_\im, S_\jm)$, for $1\leq \im \leq l, 1\leq \jm \leq i$. \\

Finally, we define the vectors: $(\bL_1)_\ell := (\bff, \bv_\ell)$, for $1\leq \ell \leq n$, $(\bL_2)_\ell := (g,q_\ell)$, for  $1\leq \ell \leq k$ and $(\bL_3)_\ell := (h,S_\ell)$, for $1\leq \ell \leq i$. 
We rewrite using the above notation the problem \eqref{fdvar} as a system of ODEs
\begin{subequations}
	\begin{alignat}{2}
		&\bA_{TT} \ddt \bT_i + \bA_{Tp} \ddt \bP_k
		+ \bA_{T\bsig} \ddt \bSig_m - \bA_{\bw T} \bW_l + \bA_{\br T}^T \bR_j = \bL_3,\label{hja5}\\
		&\bA_{\br} \bR_j - \bA_{\br T} \bT_i = 0, \label{hja6}\\
		&\bA_{pp} \ddt \bP_k + \bA_{Tp}^T \ddt \bT_i
		+ \bA_{p \bsig} \ddt \bSig_m + \bA_{\bw p}^T \bW_l = \bL_2,\label{hja3}\\
		&\bA_{\bw} \bW_l - \bA_{\bw p} \bP_k = 0,\label{hja4}\\
		&\bA_{\bsig \bsig} \bSig_m + \bA_{\bu \bsig}^T \bU_n + \bA_{p \bsig}^T \bP_k
		+ \bA_{T\bsig}^T \bT_i = 0, \label{hja1}\\
		&-\bA_{\bu \bsig} \bSig_m = \bL_1.\label{hja2}
	\end{alignat}
\end{subequations}
After rearranging, these ODE equations can be written in the form of a DAE system
\begin{equation}\label{DAE}
	\Phi \ddt X(t) + \Psi X(t) = L(t),
\end{equation}
where $X(t) := (\bP_k(t), \bSig_m(t), \bT_i(t), \bW_l(t), \bU_n(t), \bR_j(t))^T$, $L(t) := (\bL_2(t), 0, \bL_3(t), 0, \bL_1(t), 0)^T$ and
\begin{equation}
\Phi := 
\begin{pmatrix}
	\bA_{pp} & \bA_{p \bsig} & \bA_{Tp}^T & 0 & 0 & 0 \\
	0 & 0 & 0 & 0 & 0 & 0 \\
	\bA_{Tp} & \bA_{T\bsig} & \bA_{TT} & 0 & 0 & 0 \\
	0 & 0 & 0 & 0 & 0 & 0 \\
	0 & 0 & 0 & 0 & 0 & 0 \\
	0 & 0 & 0 & 0 & 0 & 0 
\end{pmatrix},
\end{equation}
and
\begin{equation}
\Psi := 
\begin{pmatrix}
	0 & 0 & 0 & \bA_{\bw p}^T & 0 & 0 \\
	\bA_{p \bsig}^T & \bA_{\bsig \bsig} & \bA_{T\bsig}^T & 0 & \bA_{\bu \bsig}^T & 0 \\
	0 & 0 & 0 & -\bA_{\bw T} & 0 & \bA_{\br T}^T \\
	-\bA_{\bw p} & 0 & 0 & \bA_{\bw \bw} & 0 & 0 \\
	0 & -\bA_{\bu \bsig} & 0 & 0 & 0 & 0 \\
	0 & 0 & -\bA_{\br T} & 0 & 0 & \bA_{\br \br}
\end{pmatrix}.
\end{equation}

From the theory of DAEs, equation \eqref{DAE} together with initial conditions \eqref{fdICs} has a solution if the matrix pencil, $s\Phi + \Psi$, is nonsingular for some $s \neq 0 $ (see \cite{brenan1995numerical}). Note that we can write $s\Phi + \Psi$ as a block $2\times 2$ matrix as follows
\begin{equation*}
	s\Phi + \Psi = 
	\begin{pmatrix}
		A & B \\
		-C & D
	\end{pmatrix},
\end{equation*}

where
\begin{equation*}
A = 
\begin{pmatrix}
	s\bA_{pp} & s\bA_{p \bsig} & s\bA_{Tp}^T \\
	\bA_{p \bsig}^T & \bA_{\bsig \bsig} & \bA_{T\bsig}^T \\
	s\bA_{Tp} & s\bA_{T\bsig} & s\bA_{TT}
\end{pmatrix},
B = 
\begin{pmatrix}
	\bA_{\bw p}^T & 0 & 0 \\
	0 & \bA_{\bu \bsig}^T & 0 \\
	-\bA_{\bw T} & 0 & \bA_{\br T}^T
\end{pmatrix},
C = 
\begin{pmatrix}
	\bA_{\bw p} & 0 & 0 \\
	0 & \bA_{\bu \bsig} & 0 \\
	& 0 & \bA_{\br T}
\end{pmatrix},
D = 
\begin{pmatrix}
	\bA_{\bw \bw} & 0 & 0 \\
	0 & 0 & 0 \\
	0 & 0 & \bA{\br \br}
\end{pmatrix}.
\end{equation*}\\

Let $\Bcal = \Scal_m \times \Pcal_k \times \Tcal_i$ and $\Ccal = \Ucal_n \times \Wcal_l \times \Rcal_j$, such that the bilinear form associated with $s\Phi + \Psi$ can be decomposed into the bilinear forms associated with each block, i.e. $\phi_A : \Bcal \times \Bcal \rightarrow \real$, $\phi_B : \Ccal \times \Bcal \rightarrow \real$, $\phi_C : \Bcal \times \Ccal \rightarrow \real$, and $\phi_D : \Ccal \times \Ccal \rightarrow \real$, where
\bse
\begin{alignat}{3}
\nonumber&\phi_A((\bsig_m, p_k, T_i),(\btau, q, S)) &&:= s(c_0 + c_r) (p_k, q) 
	+ \frac{c_r}{2\alpha} (\bid p_k, \btau) + s\frac{c_r}{2\alpha}(\bsig_m, q \bid) - sb_r(p_k,S) \\
	\nonumber&&&\qquad-sb_r(T_i,q) + (\Acal \bsig_m, \btau) + s\frac{a_r}{2\beta} (\bsig_m, S \bid) \\
	&&&\qquad\qquad+ \frac{a_r}{2\beta}(\bid T_i, \btau) + s(a_0 + a_r) (T_i, S), \\
	&\phi_B((\btau, q, S),(\bu_n, \bw_l, \br_j)) &&:= (\nabla \cdot \bw_l,q) + (\bu_n, \nabla \cdot \btau) 
	- (\etab \cdot \bw_l, S) + (\nabla \cdot \br_j,S),\\
	&\phi_{C}((\bsig_m, p_k, T_i),(\bv, \bz, \by)) &&:= (p_k,\nabla \cdot \bz) + (\nabla \cdot \bsig_m, \bv) 
	+ (T_i, \nabla \cdot \by),\\
	&\phi_D((\bu_n,\bw_l,\br_j),(\bv,\bz,\by)) &&:= (\bK^{-1} \bw_l,\bz) + (\bth^{-1} \br_j,\by).
\end{alignat}
\ese

The following Lemma will imply the invertibility of $s\Phi + \Psi$ for some $s \neq 0$.
\begin{lemma}
	For any tuples $(i,j,k,l,m,n) \geq 1$, there exists an $s \neq 0$ such that the bilinear form associated with $s\Phi + \Psi$ is strictly positive i.e.
	\begin{equation}
		\phi_A + \phi_B - \phi_C + \phi_D > 0,
	\end{equation}
	for all nonzero $(\btau, q, S) \in \Bcal$, and $(\bv, \bz, \by) \in \Ccal$.
\end{lemma}
\begin{proof}
	Denoting by $\btau = \begin{pmatrix} \tau_{11}&  \tau_{12} \\ \tau_{21} & \tau_{22} \end{pmatrix}$, 
	and using the definition of the compliance tensor \eqref{compliance}, together with the C-S, 
	Young, and triangle inequalities yields
\begin{align}\label{posdef}
	\nonumber \phi_A((\btau,q,S),&(\btau,q,S)) + \phi_B((\bv, \bz, \by),(\tau, q, S)) 
	- \phi_C((\tau, q, S),(\bv, \bz, \by)) + \phi_D((\bv, \bz, \by),(\bv, \bz, \by)) \\
	\nonumber&= s(c_0 + c_r) \norm{q}^2 + s(a_0 + a_r) \norm{S}^2 
	+ (1+s)\frac{c_r}{2\alpha} (\bid q,\btau) - 2sb_r(q,S) + (1+s)\frac{a_r}{2\beta}(\btau,S\bid)  \\
	\nonumber&\quad + (\Acal \btau,\btau)- (\etab \cdot \bz, S) + (\bK^{-1} \bz, \bz) + (\bth^{-1} \by, \by) \\
	\nonumber&\qquad\geq \left( s(c_0 + c_r - b_r) - (1+s)\frac{c_r}{2\alpha} \frac{\epsilon_1}{2} \right)\norm{q}^2 
	+ \left(s(a_0 + a_r - b_r) - (1+s)\frac{a_r}{2\beta} \frac{\epsilon_2}{2} - \frac{\gamma}{2k_m} \right)\norm{S}^2 \\
	\nonumber&\qquad\quad + \left(\frac{1}{2(\mu+\lambda)} - (1+s)\frac{c_r}{2\alpha} \frac{1}{2\epsilon_1}
	- (1+s)\frac{a_r}{2\beta} \frac{1}{2\epsilon_2} \right) \left( \norm{\btau_{11}}^2 + \norm{\btau_{22}}^2 \right) \\
	&\qquad \qquad+ \theta_m \norm{\by}^2 + \frac{k_m}{2}\norm{\bz}^2 + \frac{1}{\mu} \norm{\btau_{12}}^2.
\end{align}
	What remains is to show if there exist  parameters $\epsilon_1, \epsilon_2$, and $s$ such that the 
	following six constraints are satisfied
\begin{align}
	0 &\leq s(c_0 + c_r - b_r) - (1+s)\frac{c_r}{2\alpha} \frac{\epsilon_1}{2} \\
	0 &\leq s(a_0 + a_r - b_r) - (1+s)\frac{a_r}{2\beta} \frac{\epsilon_2}{2} - \frac{\gamma}{2k_m},\\
	0 &\leq \frac{1}{2(\mu+\lambda)} - (1+s)\frac{c_r}{2\alpha} \frac{1}{2\epsilon_1} 
	- (1+s)\frac{a_r}{2\beta} \frac{1}{2\epsilon_2},\\
	0 &< \epsilon_1, \epsilon_2, \ \textnormal{ and } \ s \neq 0.
\end{align}
	It is easily verified that the following choices are satisfactory: 
	$s = -2, \epsilon_1 = \dfrac{4\alpha}{c_r(1+s)} s(c_0 + c_r - b_r)$, 
	and $\epsilon_2 = \dfrac{4\beta}{a_r(1+s)} \left(s(a_0 + a_r - b_r) - \dfrac{\gamma}{2k_m}\right)$. 
	We use these choices in~\eqref{posdef}, and letting $\tilde{\gamma} = \dfrac{\gamma}{2k_m}$, it is inferred that
\begin{alignat}{2}
	\nonumber&\phi_A((\btau,q,S),(\btau,q,S)) + \phi_B((\bv, \bz, \by),(\tau, q, S)) 
	- \phi_C((\tau, q, S),(\bv, \bz, \by)) + \phi_D((\bv, \bz, \by),(\bv, \bz, \by)) \\
	\nonumber&\qquad \geq \frac{1}{2(\mu + \lambda)}\left(1 + \frac{1}{16(\mu + \lambda)(c_0 + c_r - b_r)}
	+ \frac{1}{16(\mu + \lambda)(a_0 + a_r - b_r + \tilde{\gamma})} \right) 
	\left( \norm{\btau_{11}}^2 + \norm{\btau_{22}}^2\right) \\
	&\qquad\qquad  +\frac{k_m}{2} \norm{\bz}^2 + \theta_m \norm{\by}^2 + \frac{1}{\mu}\norm{\tau_{12}}^2  > 0, \qquad \textnormal{for all nonzero } (\btau, q, S) \in \Bcal,\ (\bv, \bz, \by) \in \Ccal.
\end{alignat}
Thus, there exists an $s \neq 0$ such that $s\Phi + \Psi$ is nonsingular, and the equation \eqref{DAE} has a solution.
\end{proof}
\subsection{\textit{A priori} estimates}\label{subsec:apriori}

In this section, we derive \textit{a priori} estimates for the unknowns which will allow us to pass to the limit 
in problem \eqref{fdvar} by weak compactness arguments. We summarize these estimates in the following theorem.
\begin{theorem}[\textit{A priori} estimates]\label{apriori}
Under the Assumption~\ref{assumpdata}, there exists a constant $C > 0$, independent of $(i,j,k,l,m,n)\geq1$, such that
\begin{itemize}
\item[(i)]
	$\quad \norm{p_k}_{L^\infty(J;L^2(\Omega))}^2 + \norm{T_i}_{L^\infty(J;L^2(\Omega))}^2
	+ \norm{\bw_l}^2_{L^2(J;L^2(\Omega))} + \norm{\br_j}^2_{L^2(J;L^2(\Omega))} + \norm{\bsig(0)}_{\Acal}^2$
	
	$\hspace{4cm} \leq C\left( \norm{\bff}^2_{H^1(J;L^2(\Omega))}
	+ \norm{g}^2_{L^2(J;L^2(\Omega))} + \norm{h}^2_{L^2(J;L^2(\Omega))} 
	+ \norm{p_0}^2_{H^1_0(\Omega)} + \norm{T_0}^2_{H^1_0(\Omega)} \right)$,
\item[(ii)]
	$\quad \norm{\partial_t p_k}^2_{L^2(J;L^2(\Omega))} + \norm{\partial_t T_i}^2_{L^2(J;L^2(\Omega))}
	+ \norm{\bw_l}^2_{L^\infty(J;L^2(\Omega))} + \norm{\br_j}^2_{L^\infty(J;L^2(\Omega))}$
	
	$\hspace{4cm} \leq C \left(\norm{\bff}^2_{H^1(J;L^2(\Omega))} + \norm{g}^2_{L^2(J;L^2(\Omega))}
	+ \norm{h}^2_{L^2(J;L^2(\Omega))} + \norm{p_0}^2_{H^1_0(\Omega)} + \norm{T_0}^2_{H^1_0(\Omega)}\right)$,
\item[(iii)]
	$\quad \norm{\bsig_m}^2_{L^\infty(J;L^2(\Omega))} + \norm{\partial_t \bsig_m}^2_{L^2(J;L^2(\Omega))} 
	+ \norm{\bu_n}^2_{L^\infty(J;L^2(\Omega))} + \norm{\partial_t \bu_n}^2_{L^2(J;L^2(\Omega))} $
	
	$\hspace{4cm} \leq C\left( \norm{\bff}^2_{H^1(J;L^2(\Omega))} + \norm{g}^2_{L^2(J;L^2(\Omega))} 
	+ \norm{h}^2_{L^2(J;L^2(\Omega))} + \norm{p_0}^2_{H_0^1(\Omega)} + \norm{T_0}^2_{H_0^1(\Omega)}\right)$,
\item[(iv)]
	$\quad \norm{\bw_l}_{L^2(J;H(\divr, \Omega))}^2 + \norm{\br_j}_{L^2(J;H(\divr, \Omega))}^2 
	+ \norm{\bsig_m}^2_{L^2(J;H_s(\divr,\Omega))}$
	
	$\hspace{4cm} \leq C \left(\norm{\bff}^2_{H^1(J;L^2(\Omega))} + \norm{g}^2_{L^2(J;L^2(\Omega))}
	+ \norm{h}^2_{L^2(J;L^2(\Omega))} + \norm{p_0}^2_{H^1_0(\Omega)} + \norm{T_0}^2_{H^1_0(\Omega)}\right).$
\end{itemize}
\end{theorem}
\begin{proof}
By Thomas' Lemma~\ref{thomas} there exist $\tilde{\bsig} \in H^1(J;\Scal_m)$ such that $-\nabla \cdot \tilde{\bsig}(\cdot,t) = \bu_n(\cdot, t)$ on $\Omega$ for $t \in J$, and with $\norm{\tilde{\bsig}(t)} \leq C \norm{\bu_n(t)}$. 
Thus, we set $\btau_\ell = \tilde{\bsig}(t)$ in \eqref{fd1} and obtain
\begin{align}
	\nonumber \norm{\bu_n}^2 &= -(\bu_n, \nabla \cdot \tilde{\bsig}) = (\Acal \bsig_m, \tilde{\bsig}) 
	+ \frac{c_r}{2\alpha} (\bid p_k,\tilde{\bsig}) + \frac{a_r}{2\beta}(\bid T_i, \tilde{\bsig}),\\
	&\leq \left(\frac{1}{2\mu}\norm{\bsig_m} + \frac{c_r}{2\alpha}\norm{p_k} 
	+ \frac{a_r}{2\beta}\norm{T_i}\right)\norm{\tilde{\bsig}} 
	\leq \left(\frac{1}{2\mu}\norm{\bsig_m} + \frac{c_r}{2\alpha}\norm{p_k} 
	+ \frac{a_r}{2\beta}\norm{T_i}\right)C\norm{\bu_n}, 
\end{align}
which implies
\begin{equation}\label{b1}
	\norm{\bu_n}^2 \leq C\left(\norm{\bsig_m}^2 + \norm{p_k}^2 + \norm{T_i}^2\right),
\end{equation}
for some constant $C > 0$ depending on the coefficients, domain and spatial dimension. Next, we take $\btau_\ell = \bsig_m$ in \eqref{fd1} and $\bv_\ell = \bu_n$ in \eqref{fd2}, and add the resulting equations together to obtain
\begin{equation}
	\norm{\bsig_m}^2_{\Acal} = -\frac{c_r}{2\alpha}(\bid p_k, \bsig_m) 
	- \frac{a_r}{2\beta} (\bid T_i, \bsig_m) + (\bff, \bu_n).
\end{equation}
Applying the C-S and Young inequalities together with the above estimate \eqref{b1} yields
\begin{align}
	\nonumber \norm{\bsig_m}^2_{\Acal} &\leq \frac{c_r}{2\alpha} \left( \frac{1}{2\epsilon_1} \norm{p_k}^2 
	+ \frac{\epsilon_1}{2} \norm{\bsig_m}^2 \right)
	+ \frac{a_r}{2\beta} \left( \frac{1}{2\epsilon_2} \norm{T_i}^2 + \frac{\epsilon_2}{2} \norm{\bsig_m}^2 \right) 
	+ \frac{1}{2\epsilon_3} \norm{\bff}^2 + \frac{\epsilon_3}{2} \norm{\bu_n}^2\\
	\nonumber&\leq \left( \frac{\alpha}{2} \epsilon_1 + \frac{\beta}{2} \epsilon_2 
	+ C(\mu + \lambda)\epsilon_3 \right) \norm{\bsig_m}^2_{\Acal} 
	+ \left(\frac{c_r}{4\alpha \epsilon_1} + C\frac{\epsilon_3}{2} \right) \norm{p_k}^2\\
	&\hspace{1cm}+ \left(\frac{a_r}{4\beta \epsilon_2} + C\frac{\epsilon_3}{2} \right) \norm{T_i}^2 
	+ \frac{1}{2\epsilon_3} \norm{\bff}^2.
\end{align}
Choosing suitable values for the epsilons, i.e., $\epsilon_1 = \dfrac{1}{3\alpha}$, $\epsilon_2 = \dfrac{1}{3\beta}$, and $\epsilon_3 = \dfrac{1}{6C(\mu + \lambda)}$, we obtain	
\begin{equation}\label{bb1}
	\norm{\bsig_m}^2_{\Acal} \leq \left(\frac{3}{2} c_r + \frac{1}{6(\mu + \lambda)} \right)\norm{p_k}^2
	+ \left(\frac{3}{2} a_r + \frac{1}{6(\mu + \lambda)} \right)\norm{T_i}^2 + C\norm{\bff}^2.
\end{equation}
It then follows immediately that
\begin{equation}\label{bb2}
	\norm{\bu_n}^2 \leq C\left(\norm{p_k}^2 + \norm{T_i}^2 + \norm{\bff}^2\right).
\end{equation}
Take now $\tilde{\bsig} \in L^2(J;\Scal_m)$ such that $-\nabla \cdot \tilde{\bsig}(\cdot,t) = \partial_t\bu_n (\cdot,t)$ on $\Omega$, for $t \in J$, and with $\norm{\tilde{\bsig}(t)} \leq C \norm{\partial_t \bu_n(t)}$. Then, by differentiating equation \eqref{fd1} with respect to time, and setting $\btau_\ell = \tilde{\bsig}$, we get in the same way as before
\begin{equation}\label{b2}
	\norm{\partial_t \bu_n}^2 \leq C\left(\norm{\partial_t \bsig_m}^2 + \norm{\partial_t p_k}^2 + \norm{\partial_t T_i}^2\right).
\end{equation}
We continue by differentiating equations \eqref{fd1} and \eqref{fd2} with respect to time, and take $\partial_t \bsig_m$ and $\partial_t \bu_n$ as test functions, respectively, and get analogously
\begin{equation}\label{bb3}
	\norm{\partial_t \bsig_m}^2_{\Acal} \leq \left(\frac{3}{2} c_r + \frac{1}{6(\mu + \lambda)} \right)\norm{\partial_t p_k}^2
	+ \left(\frac{3}{2} a_r + \frac{1}{6(\mu + \lambda)} \right)\norm{\partial_t T_i}^2 + C\norm{\partial_t \bff}^2,
\end{equation}
and 
\begin{equation}\label{bb4}
	\norm{\partial_t \bsig_m}^2 \leq C\left(\norm{\partial_t p_k}^2 + \norm{\partial_t T_i}^2
	+ \norm{\partial_t \bff}^2\right),
\end{equation}
where the constants $C > 0$ depends on the coefficients, domain, and spatial dimension.
Next, we take $\partial_t \bsig_m$, $p_k$, $\bw_l$, $T_i$ and $\br_j$ as a test functions in \eqref{fd1}, \eqref{fd3}, \eqref{fd4}, \eqref{fd5} and \eqref{fd6}, respectively. We differentiate then~\eqref{fd2} with respect to time, 
and take $\bu_n$ as a test function. Adding together the resulting equations yields
\begin{align}
	\nonumber &(c_0 + c_r)(\partial_t p_k, p_k) + (a_0 + a_r)(\partial_t T_i, T_i) + (\bK^{-1} \bw_l, \bw_l) 
	+ (\bth^{-1} \br_j, \br_j) \\
	&= (\Acal \bsig_m, \partial_t \bsig_m) + b_r(\partial_t T_i, p_k) + b_r(\partial_t p_k, T_j) + (\etab \cdot \bw_l, T_i)
	- (\partial_t \bff, \bu_n) + (g,p_k) + (h,T_i).
\end{align}
Using the properties of $\bK$ and $\bth$, in addition to the C-S and Young inequalities yields
\begin{align}
	\nonumber \left(c_0 + c_r - b_r \right) \frac{1}{2}\ddt \norm{p_k}^2 
	&+ \left(a_0 + a_r - b_r \right) \frac{1}{2} \ddt \norm{T_i}^2 
	+ \left(k_m - \gamma \frac{1}{2\epsilon}\right) \norm{\bw_l}^2 + \theta_m \norm{\br_j}^2 \\
	&\leq \frac{1}{2} \left(\ddt \norm{\bsig_m}^2_{\Acal}
	+ (\epsilon +1) \norm{T_i}^2    
	+ \norm{\bu_n}^2 + \norm{p_k}^2 + \norm{\partial_t \bff}^2 + \norm{g}^2 + \norm{h}^2\right).
\end{align}
Choosing $\epsilon = \dfrac{\gamma}{k_m}$, integrating from $0$ to $t$ and 
substituting the inequalities \eqref{b1} and \eqref{bb1}, we deduce
\begin{align}
	\nonumber &\left(c_0 - \frac{c_r}{2} - b_r - \frac{1}{6(\mu + \lambda)}  \right) \norm{p_k(t)}^2 
	+ \left(a_0 - \frac{a_r}{2} - b_r - \frac{1}{6(\mu + \lambda)} \right) \norm{T_i(t)}^2
	+ \int_0^t \bigg(k_m \norm{\bw_l(\tau)}^2 + \theta_m \norm{\br_j(\tau)}^2 \bigg)\dtau \\
	\nonumber&\hspace{1cm}\leq C\int_0^t \left(\norm{p_k(\tau)}^2 + \norm{T_i(\tau)}^2 \right) \dtau 
	- \norm{\bsig_m(0)}^2_{\Acal}\\
	&\hspace{2cm}+C\left( \norm{\bff}^2_{H^1(J;L^2(\Omega))}
	+ \norm{g}^2_{L^2(J;L^2(\Omega))} + \norm{h}^2_{L^2(J;L^2(\Omega))} 
	+ \norm{p_k(0)}^2 + \norm{T_i(0)}^2 \right).
\end{align}
Since from \eqref{fdICs} we have
\begin{equation}
	\norm{T_i(0)}^2 \leq \norm{T_0}^2 \text{ and } \norm{p_k(0)}^2 \leq \norm{p_0}^2,
\end{equation}
we obtain the first estimate $(i)$ using Gr\"onwall's inequality, i.e.
\begin{align}\label{bb5}
	\nonumber&\norm{p_k}_{L^\infty(J;L^2(\Omega))}^2 + \norm{T_i}_{L^\infty(J;L^2(\Omega))}^2
	+ \norm{\bw_l}^2_{L^2(J;L^2(\Omega)} + \norm{\br_j}^2_{L^2(J;L^2(\Omega)} + \norm{\bsig_m(0)}_{\Acal}^2\\
	&\hspace{3cm} \leq C\left( \norm{\bff}^2_{H^1(J;L^2(\Omega))}
	+ \norm{g}^2_{L^2(J;L^2(\Omega))} + \norm{h}^2_{L^2(J;L^2(\Omega))} 
	+ \norm{p_0}^2 + \norm{T_0}^2 \right),
\end{align}
where the constant $C>0$ depends on $\etab$, the coefficients, domain and spatial dimension.
For the second estimate, we differentiate \eqref{fd1}, \eqref{fd2}, \eqref{fd4} and \eqref{fd6} 
with respect to time and use $\partial_t \bsig_m, \partial_t \bu_n, \bw_l$ and $\br_j$ as test functions, 
respectively. In \eqref{fd3} and \eqref{fd5}, we use $\partial_t p_k$ and $\partial_t T_i$ as test functions, respectively. Summing the resulting equations yields
\begin{align}
	\nonumber &(c_0 + c_r)\norm{\partial_t p_k}^2 + (a_0 + a_r)\norm{\partial_t T_i}^2 
	+ (\bK^{-1} \partial_t \bw_l, \bw_l) + (\bth^{-1} \partial_t \br_j, \br_j) \\
	&\qquad= \norm{\partial_t \bsig_m}^2_{\Acal} + 2b_r(\partial_t T_i, \partial_t p_k) + (\etab \cdot \bw_l, \partial_t T_i) 
	- (\partial_t \bff, \partial_t \bu_n) + (g, \partial_t p_k) + (h, \partial_t T_i).
\end{align}
By applying the C-S and Young inequalities, and substituting the estimates \eqref{b2} and \eqref{bb3}, we deduce
\begin{align}
	\nonumber&\left( c_0 - \frac{c_r}{2} - b_r - \frac{1}{6(\mu + \lambda)} - \frac{\epsilon_2}{2}\right) 
	\norm{\partial_t p_k}^2 \\
	\nonumber&\hspace{1cm} + \left( a_0 - \frac{a_r}{2} - b_r - \frac{1}{6(\mu + \lambda)} - \frac{\epsilon_4}{2} - \frac{\epsilon_3}{2} \right) 
	\norm{\partial_t T_i}^2+ \frac{k_m}{2} \ddt \norm{\bw_l}^2 + \frac{\theta_m}{2} \ddt \norm{\br_j}^2 \\
	&\hspace{2cm} \leq \frac{\epsilon_1}{2} C \left( \norm{\partial_t p_k}^2 + \norm{\partial_t T_i}^2 + \norm{\partial_t \bff}^2\right)
	+ \gamma \frac{1}{2\epsilon_4} \norm{\bw_l}^2 + \frac{1}{2\epsilon_1} \norm{\partial_t \bff}^2 
	+ \frac{1}{2\epsilon_2} \norm{g}^2 
	+ \frac{1}{2\epsilon_3} \norm{h}^2.
\end{align}
Choosing suitable values for the epsilons, i.e. $\epsilon_1 = \dfrac{\alpha \beta}{C(\mu + \lambda)}$, $\epsilon_2 = \dfrac{\alpha \beta}{\mu + \lambda}$, $\epsilon_3 = \dfrac{\alpha \beta}{2(\mu + \lambda)}$, and $\epsilon_4 = \dfrac{\alpha \beta}{2(\mu + \lambda)}$, 
we infer
\begin{align}\label{helpbound}
	\nonumber \left( c_0 - \frac{c_r}{2} - b_0 - \frac{1}{6(\mu + \lambda)} \right) \norm{\partial_t p_k}^2
	&+ \left( a_0 - \frac{a_r}{2} - b_0 - \frac{1}{6(\mu + \lambda)} \right) \norm{\partial_t T_i}^2 
	+ \frac{k_m}{2} \ddt \norm{\bw_l}^2 + \frac{\theta_m}{2} \ddt \norm{\br_j}^2 \\
	&\hspace{1cm} \leq C\left( \norm{\bw_l}^2 + \norm{\partial_t \bff}^2 + \norm{g}^2 + \norm{h}^2 \right).
\end{align}
Simplifying the above expression, integrating over $(0,t)$ and using the initial conditions yields
\begin{align}
	\nonumber&\norm{\bw_l(t)}^2 + \norm{\br_j(t)}^2 + \int_0^t \left(\norm{\partial_t p_k(\tau)}^2 
	+ \norm{\partial_t T_i(\tau)}^2\right) \dtau \\
	&\leq C\int_0^t \norm{\bw_l(\tau)}^2 \dtau 
	+ C\left(\norm{\partial_t \bff}^2_{L^2(J;L^2(\Omega))} + \norm{g}^2_{L^2(J;L^2(\Omega))}
	+ \norm{h}^2_{L^2(J;L^2(\Omega))} + \norm{\bw_l(0)}^2 + \norm{\br_j(0)}^2\right).\label{somebound}
\end{align}
It remains to provide estimates for $\norm{\bw_l(0)}^2$ and $\norm{\br_j(0)}^2$. To this end, take $\bw_l$ as a test function in equation \eqref{fd4}, and set $t = 0$. This gives
\begin{equation}
	(\bK^{-1} \bw_l(0), \bw_l(0)) = (p_k(0), \nabla \cdot \bw_l(0)),
\end{equation}
which holds true for any $k,l \geq 1$. Use now the properties of $\bK$ to bound the left-hand side, tend $k \rightarrow \infty$ and then integrate by parts in the right-hand side to obtain
\begin{equation}
k_m \norm{\bw_l(0)}^2 \leq (p_0, \nabla \cdot \bw_l(0)) = - (\nabla p_0, \bw_l(0)) \leq \norm{\nabla p_0} \norm{\bw_l(0)}.
\end{equation}
Thus, we have
\begin{equation}\label{winitbound}
	\norm{\bw_l(0)}^2 \leq C \norm{p_0}^2_{H^1_0(\Omega)}.
\end{equation}
Similarly, using \eqref{fd6}, we obtain
\begin{equation}\label{rinitbound}
	\norm{\br_j(0)}^2 \leq C\norm{T_0}^2_{H^1_0(\Omega)}.
\end{equation}
Taking now \eqref{winitbound} and \eqref{rinitbound} in \eqref{somebound}, and applying Gr\"onwall's lemma~\ref{gr}, we obtain the 
second  estimate $(ii)$, i.e.
\begin{align}\label{bb6}
	\nonumber \norm{\partial_t p_k}^2_{L^2(J;L^2(\Omega))} &+ \norm{\partial_t T_i}^2_{L^2(J;L^2(\Omega))}
	+ \norm{\bw_l}^2_{L^\infty(J;L^2(\Omega))} + \norm{\br_j}^2_{L^\infty(J;L^2(\Omega))}\\
	&\leq C \left(\norm{\bff}^2_{H^1(J;L^2(\Omega))} + \norm{g}^2_{L^2(J;L^2(\Omega))}
	+ \norm{h}^2_{L^2(J;L^2(\Omega))} + \norm{p_0}^2_{H^1_0(\Omega)} + \norm{T_0}^2_{H^1_0(\Omega)}\right),
\end{align}
where the constant $C>0$ depends on $\etab$, the coefficients, domain and spatial dimension.
Now we sum the estimates \eqref{bb1}, \eqref{bb2}, \eqref{bb3}, and \eqref{bb4}, and substitute the estimates \eqref{bb5} and \eqref{bb6}, to obtain $(iii)$, i.e.
\begin{align}\label{sigubound}
	\nonumber&\norm{\bsig_m}^2_{L^\infty(J;L^2(\Omega))} + \norm{\partial_t \bsig_m}^2_{L^2(J;L^2(\Omega))} 
	+ \norm{\bu_n}^2_{L^\infty(J;L^2(\Omega))} + \norm{\partial_t \bu_n}^2_{L^2(J;L^2(\Omega))} \\
	&\hspace{3cm} \leq C\left( \norm{\bff}^2_{H^1(J;L^2(\Omega))} + \norm{g}^2_{L^2(J;L^2(\Omega))} 
	+ \norm{h}^2_{L^2(J;L^2(\Omega))} + \norm{p_0}^2_{H_0^1(\Omega)} + \norm{T_0}^2_{H_0^1(\Omega)}\right).
\end{align}
It remains to obtain the estimate $(iv)$, for which we need just to bound the divergences. Since $\nabla \cdot \br_j(t) \in L^2(\Omega)$ for $t \in J$, we can write $\nabla \cdot \br_j(t) = \sum_{\ell = 1}^\infty \xi_\ell(t) S_\ell$, for some functions $\xi_\ell(t) \in \real$. Now, we multiply equation \eqref{fd5} with $\xi_\ell$, sum over $\ell = 1,..,i$ and use the C-S and Young inequalities to obtain
\begin{align}\label{predivr}
	\nonumber&(\nabla \cdot \br_j, \sum_{\ell = 1}^{i} \xi_\ell S_\ell) 
	= (h,\sum_{\ell = 1}^{i} \xi_\ell S_\ell)
	- (a_0 + a_r) (\partial_{t} T_i, \sum_{\ell = 1}^{i} \xi_\ell S_\ell)
	- \frac{a_r}{2\beta} (\partial_{t} \bsig_l, \sum_{\ell = 1}^{i} \xi_\ell S_\ell)
	+ b_r(\partial_{t} p_k, \sum_{\ell = 1}^{i} \xi_\ell S_\ell)
	+ (\etab \cdot \bw_l, \sum_{\ell = 1}^{i} \xi_\ell S_\ell) \\
	&\hspace{1cm}\leq \frac{1}{2} \bigg(\norm{\sum_{\ell = 1}^{i} \partial_t \xi_\ell q_\ell}^2 + 5\norm{h}^2 
	+ 5(a_0 + a_r)^2 \norm{\partial_{t} T_i}^2 
	+ \frac{5a_r^2}{4\beta^2} \norm{\partial_{t} \bsig_m}^2 
	+ 5b_r^2 \norm{\partial_{t} p_k}^2 + 5\gamma \norm{\bw_l}^2 \bigg).
\end{align}
Using \eqref{bb3}, integrating in time and using \eqref{bb6} we get
\begin{align}\label{bb7}
	\int_0^{T_f} (\nabla \cdot \br_j, \sum_{\ell = 1}^{i} \xi_{\ell} S_{\ell}) \dt
	\nonumber&\leq \frac{1}{2} \int_0^{T_f} \norm{\sum_{\ell = 1}^{i} \xi_{\ell} S_{\ell}}^2\dt \\
	&+ C\left(\norm{\bff}^2_{H^1(J;L^2(\Omega))} + \norm{g}^2_{L^2(J;L^2(\Omega))} 
	+ \norm{h}^2_{L^2(J;L^2(\Omega))} + \norm{p_0}^2_{H_0^1(\Omega)} + \norm{T_0}^2_{H_0^1(\Omega)} \right).
\end{align}
Finally, tend $i \rightarrow \infty$ and obtain
\begin{equation}\label{divbound1}
	\norm{\nabla \cdot \br_j}^2_{L^2(J;L^2(\Omega))} 
	\leq C\left( \norm{\bff}^2_{H^1(J;L^2(\Omega))}
	+ \norm{g}^2_{L^2(J;L^2(\Omega))} + \norm{h}^2_{L^2(J;L^2(\Omega))} + \norm{p_0}^2_{H_0^1(\Omega)} + \norm{T_0}^2_{H_0^1(\Omega)}\right).
\end{equation}
From equations \eqref{fd3} and \eqref{fd2} we obtain using the same technique
\begin{equation}\label{divbound3}
	\norm{\nabla \cdot \bw_l}^2_{L^2(J;L^2(\Omega))} \leq C\left( \norm{\bff}^2_{H^1(J;L^2(\Omega))}
	+ \norm{g}^2_{L^2(J;L^2(\Omega))} + \norm{h}^2_{L^2(J;L^2(\Omega))} + \norm{p_0}^2_{H_0^1(\Omega)} + \norm{T_0}^2_{H_0^1(\Omega)} \right),
\end{equation}
and 
\begin{equation}\label{divbound2}
	\norm{\nabla \cdot \bsig_m}^2_{L^2(J;L^2(\Omega))} \leq C \norm{\bff}^2_{L^2(J;L^2(\Omega))},
\end{equation}
where the constants $C>0$ depends on $\etab$, the coefficients, domain and spatial dimension. Combining now the estimates  \eqref{divbound1}--\eqref{divbound3} with $(i)$ and $(iii)$, we get the estimate $(iv)$. This ends the proof.
\end{proof}
The following estimates proves that the solution has improved regularity given some additional regularity on the data. We state the result as a lemma:
\begin{lemma}[Estimates for improved regularity]\label{improvedregularity}
	Assume that $\bff \in H^2(J;L^2(\Omega))$ and $g,h \in H^1(J;L^2(\Omega))$. Then there exists a constant $C>0$ independent of $(i,j,k,l,m,n)$ such that
	\begin{itemize}
	\item[(i)]
	$\quad \norm{p_k}^2_{W^{1,\infty}(J;L^2(\Omega))} + \norm{T_i}^2_{W^{1,\infty}(J;L^2(\Omega))}
	+ \norm{\bw_l}^2_{H^1(J;L^2(\Omega))} + \norm{\br_j}^2_{H^1(J;L^2(\Omega))} + \norm{\partial_t \bsig_m(0)}^2$
	
	$\hspace{2cm} \leq C \left( \norm{\bff}^2_{H^2(J;L^2(\Omega))} + \norm{g}^2_{H^1(J;L^2(\Omega))}
	+ \norm{h}^2_{H^1(J;L^2(\Omega))} + \norm{p_0}^2_{H_0^1(\Omega)} + \norm{T_0}^2_{H_0^1(\Omega)}\right)$,

	\item[(ii)]
	$\quad \norm{\bsig_m}^2_{W^{1,\infty}(J;L^2(\Omega))} + \norm{\bu_n}^2_{W^{1,\infty}(J;L^2(\Omega))}$  
	
	$\hspace{2cm} \leq C \left( \norm{\bff}^2_{H^2(J;L^2(\Omega))} + \norm{g}^2_{H^1(J;L^2(\Omega))}
	+ \norm{h}^2_{H^1(J;L^2(\Omega))} + \norm{p_0}^2_{H_0^1(\Omega)} + \norm{T_0}^2_{H_0^1(\Omega)}\right)$,
	\item[(iii)]
	$\quad \norm{\bw_l}^2_{L^\infty(J;H(\divr,\Omega))} + \norm{\br_j}^2_{L^\infty(J;H(\divr,\Omega))}
	+ \norm{\bsig_m}^2_{L^\infty(J;H_s(\divr,\Omega))}$
	
	$\hspace{2cm} \leq C \left( \norm{\bff}^2_{H^2(J;L^2(\Omega))} + \norm{g}^2_{H^1(J;L^2(\Omega))}
	+ \norm{h}^2_{H^1(J;L^2(\Omega))} + \norm{p_0}^2_{H_0^1(\Omega)} + \norm{T_0}^2_{H_0^1(\Omega)}\right)$.
	\end{itemize}
\end{lemma}
\begin{proof}
We begin by differentiating equations \eqref{fd1}, \eqref{fd3}, \eqref{fd4}, \eqref{fd5} and \eqref{fd6} with respect to time, and take $\partial_{tt} \bsig_m$, $\partial_t p_k$, $\partial_t \bw_l$, $\partial_t T_i$ and $\partial_t \br_j$ as a test functions respectively. Then, we differentiate \eqref{fd2} twice with respect to time, and take $\partial_t \bu_n$ as a test function. Adding together the resulting equations yields
\begin{align}
	\nonumber &(c_0 + c_r)\frac{1}{2} \ddt \norm{\partial_t p_k}^2 + (a_0 + a_r)\frac{1}{2} \ddt \norm{\partial_tT_i}^2 
	+ (\bK^{-1} \partial_t\bw_l, \partial_t\bw_l) 
	+ (\bth^{-1} \partial_t\br_j, \partial_t\br_j) \\
	\nonumber&= \frac{1}{2} \ddt \norm{\partial_t\bsig_m}_{\Acal} + b_r\ddt(\partial_t T_i, \partial_tp_k)  
	+ (\etab \cdot \partial_t \bw_l, \partial_tT_i) \\
	&\hspace{6cm}- (\partial_{tt} \bff, \partial_t\bu_n) + (\partial_tg,\partial_tp_k) + (\partial_th,\partial_tT_i).
\end{align}
Using the properties of $\bK$ and $\bth$, in addition to the C-S and Young inequalities, we get
\begin{align}\label{wrinitpre}
	\nonumber &\left(c_0 + c_r - b_r \right)\frac{1}{2}\ddt \norm{\partial_t p_k}^2 
	+ \left(a_0 + a_r - b_r \right) \frac{1}{2} \ddt \norm{\partial_t T_i}^2 
	+ \frac{k_m}{2} \norm{\partial_t \bw_l}^2 + \theta_m \norm{\partial_t \br_j}^2 \\
	&\qquad \leq \frac{1}{2}\bigg(\ddt \norm{\partial_t \bsig_m}^2_{\Acal}
	+ \frac{\gamma}{k_m} \norm{\partial_tT_i}^2 + \norm{\partial_tp_k}^2 + \norm{\partial_t\bu_n}^2
	+ \norm{\partial_{tt} \bff}^2 
	 + \norm{\partial_tg}^2 + \norm{\partial_th}^2\bigg).
\end{align}
By integrating over  $(0,t)$, using the initial conditions
and substituting the inequalities \eqref{bb3} and \eqref{bb4}, it is inferred that
\begin{align}\label{lbound1}
	\nonumber &\left(c_0 - \frac{c_r}{2} - b_r - \frac{1}{6(\mu + \lambda)}  \right) \norm{\partial_tp_k(t)}^2 
	+ \left(a_0 - \frac{a_r}{2} - b_r - \frac{1}{6(\mu + \lambda)} \right) \norm{\partial_tT_i(t)}^2\\
	\nonumber&\hspace{1cm}+ \int_0^t \bigg(k_m \norm{\partial_t\bw_l(\tau)}^2 + \theta_m \norm{\partial_t\br_j(\tau)}^2 \bigg)\dtau 
	+ \norm{\partial_t\bsig_m(0)}_{\Acal}^2\\
	\nonumber&\hspace{2cm} \leq C \int_0^t \left( \norm{\partial_tp_k(\tau)}^2 + \norm{\partial_tT_i(\tau)}^2  \right) \dtau \\
	&\hspace{2.5cm}+C\left(\norm{\bff}^2_{H^2(J;L^2(\Omega))} + \norm{g}^2_{H^1(J;L^2(\Omega))} 
	+ \norm{h}^2_{H^1(J;L^2(\Omega))} + \norm{\partial_t p(0)}^2 + \norm{\partial_tT(0)}^2\right).
\end{align}
We proceed to bound $\norm{\partial_t p_k(0)}$ and $\norm{\partial_t T_i(0)}$. To this end, we discard the terms under the time differential on the left-hand side of \eqref{helpbound} and set $t=0$ to obtain
\begin{align}
	\nonumber &\left( c_0 - \frac{c_r}{2} - b_0 - \frac{1}{6(\mu + \lambda)} \right) \norm{\partial_t p_k(0)}^2
	+ \left( a_0 - \frac{a_r}{2} - b_0 - \frac{1}{6(\mu + \lambda)} \right) \norm{\partial_t T_i(0)}^2 \\
	&\qquad\qquad\qquad\qquad\hspace{2.5cm}\leq C\left(\norm{\bw_l(0)}^2 + \norm{\partial_t \bff(0)}^2 + \norm{g(0)}^2 + \norm{h(0)}^2 \right).
\end{align}
Using \eqref{winitbound} to bound the initial value of the Darcy flux yields
\begin{align}\label{lbound2}
	\norm{\partial_t p_k(0)}^2 + \norm{\partial_t T_i(0)}^2 
	\leq C\left(\norm{p_0}^2_{H^1_0(\Omega)} + \norm{\bff}^2_{H^2(J;L^2(\Omega))} 
	+ \norm{g}^2_{H^1(J;L^2(\Omega))} + \norm{h}^2_{H^1(J;L^2(\Omega))} \right).
\end{align}
Now we substitute this in \eqref{lbound1}, using also $(i)$ from Theorem \ref{apriori} and apply Gr\"onwall's Lemma~\ref{gr} to obtain 
\begin{align}\label{addreg1}
	\nonumber&\norm{\partial_tp_k}^2_{L^\infty(J;L^2(\Omega))} + \norm{\partial_tT_i}^2_{L^\infty(J;L^2(\Omega))} 
	+ \norm{\partial_t\bw_l}^2_{L^2(J;L^2(\Omega))} + \norm{\partial_t\br_j}^2_{L^2(J;L^2(\Omega))}  
	+ \norm{\partial_t\bsig_m(0)}_{\Acal}^2\\
	&\hspace{2cm}\leq C\left( \norm{\bff}^2_{H^2(J;L^2(\Omega))} + \norm{g}^2_{H^1(J;L^2(\Omega))} 
	+ \norm{h}^2_{H^1(J;L^2(\Omega))} + \norm{p_0}^2_{H_0^1(\Omega)} + \norm{T_0}^2_{H_0^1(\Omega)} \right).
\end{align}
Summing with $(i)$ from Theorem \ref{apriori} produces the estimate $(i)$.
We continue by summing \eqref{bb3} and \eqref{bb4}, and combine with \eqref{addreg1} to obtain
\begin{align}\label{siguimproved}
	\nonumber&\norm{\partial_{t} \bsig_m}^2_{L^\infty(J;L^2(\Omega))}
	 + \norm{\partial_{t} \bu_n}^2_{L^\infty(J;L^2(\Omega))} \\
	 &\hspace{2cm}\leq C\left( \norm{\bff}^2_{H^2(J;L^2(\Omega))} + \norm{g}^2_{H^1(J;L^2(\Omega))}  
	+ \norm{h}^2_{H^1(J;L^2(\Omega))} + \norm{p_0}^2_{H_0^1(\Omega)} + \norm{T_0}^2_{H_0^1(\Omega)}\right).
\end{align}
Summing the above with $(iii)$ from Theorem \ref{apriori} produces the estimate $(ii)$.
Going back to the estimate \eqref{predivr}, we now substitute in the right-hand side with \eqref{addreg1} and \eqref{siguimproved}, let $i \rightarrow \infty$ to obtain
\begin{align}\label{addregdivbound1}
	\nonumber &\norm{\nabla \cdot \br_j}^2_{L^\infty(J;L^2(\Omega))} \\
	&\hspace{2cm}\leq C\left( \norm{\bff}^2_{H^2(J;L^2(\Omega))}
	+ \norm{g}^2_{H^1(J;L^2(\Omega))} + \norm{h}^2_{H^1(J;L^2(\Omega))} 
	+ \norm{p_0}^2_{H_0^1(\Omega)} + \norm{T_0}^2_{H_0^1(\Omega)}\right).
\end{align}
From equations \eqref{fd3} and \eqref{fd2} we obtain using the same technique
\begin{align}\label{addregdivbound2}
	\nonumber &\norm{\nabla \cdot \bw_l}^2_{L^\infty(J;L^2(\Omega))} \\
	&\hspace{2cm}\leq C\left( \norm{\bff}^2_{H^2(J;L^2(\Omega))}
	+ \norm{g}^2_{H^1(J;L^2(\Omega))} + \norm{h}^2_{H^1(J;L^2(\Omega))} 
	+ \norm{p_0}^2_{H_0^1(\Omega)} + \norm{T_0}^2_{H_0^1(\Omega)} \right),
\end{align}
and 
\begin{equation}\label{addregdivbound3}
	\norm{\nabla \cdot \bsig_m}^2_{L^\infty(J;L^2(\Omega))} \leq C \norm{\bff}^2_{L^\infty(J;L^2(\Omega))}.
\end{equation}
Summing the estimates  \eqref{addregdivbound1}--\eqref{addregdivbound3} and combining with $(ii)$ and $(iii)$ from Theorem \ref{apriori} produces the estimate $(iii)$. This ends the proof.
\end{proof}
\subsection{End of the proof of Theorem \ref{linWP}:} 
The proof of the first part of Theorem~\ref{linWP} follows the steps below: \\
$\bullet$ Lemma~\ref{apriori} implies that for the sequences $\{\bsig_m \}_0^\infty$, $\{ \bu_n \}_0^\infty$, 
$\{ p_k \}_0^\infty$, $\{ \bw_l \}_0^\infty$, $\{ T_i \}_0^\infty$ and $\{ \br_j \}_0^\infty$ defined by 
\eqref{fdvar}:  $\{\bsig_m \}_0^\infty$ is bounded in $L^\infty(J;H_s(\divr,\Omega)) \cap H^1(J;L^2(\Omega))$, 
$\{ \bu_n \}_0^\infty$ is bounded in $H^1(J;L^2(\Omega))$, 
$\{ p_k \}_0^\infty$ is bounded in $H^1(J;L^2(\Omega))$, 
$\{ \bw_l \}_0^\infty$ is bounded in $L^2(J;H(\divr,\Omega)) \cap L^\infty(J; L^2(\Omega))$, 
$\{ T_i \}_0^\infty$ is bounded in $H^1(J;L^2(\Omega))$, 
and $\{ \br_j \}_0^\infty$ is bounded in $L^2(J;H(\divr,\Omega)) \cap L^\infty(J; L^2(\Omega))$.

By the weak compactness properties of the spaces there exists subsequences (denoted the same way as before) and functions $\bsig \in L^\infty(J;H_s(\divr,\Omega)) \cap H^1(J;L^2(\Omega))$, $\bu \in H^1(J;L^2(\Omega))$, $p \in H^1(J;L^2(\Omega))$, $\bw \in L^2(J;H(\divr,\Omega)) \cap L^\infty(J; L^2(\Omega))$, $T \in H^1(J;L^2(\Omega))$, and $\br \in L^2(J;H(\divr,\Omega)) \cap L^\infty(J; L^2(\Omega))$, such that
\begin{itemize}
	\item $T_i \rightharpoonup T \textnormal{ in }  H^1(J;L^2(\Omega))$,
	\item $\br_j \rightharpoonup \br \textnormal{ in }  L^2(J;H(\divr,\Omega))$,
	\item $p_k \rightharpoonup p \textnormal{ in }  H^1(J;L^2(\Omega))$,
	\item $\bw_l \rightharpoonup \bw \textnormal{ in }  L^2(J;H(\divr,\Omega))$,
	\item $\bsig_m \rightharpoonup \bsig \textnormal{ in } L^2(J;H_s(\divr,\Omega))$, 
	\item $\partial_t \bsig_m \rightharpoonup \partial_t \bsig \textnormal{ in } L^2(J;L^2(\Omega))$,
	\item $\bu_n \rightharpoonup \bu \textnormal{ in }  H^1(J;L^2(\Omega))$.
\end{itemize}
In order to pass to the limit in problem \eqref{fdvar}, we fix a tuple $(i,j,k,l,m,n) \geq 1$ and take $(S,\by, q,\bz,\btau,\bv) \in C^1(J;\Tcal_{i} \times \Rcal_{j} \times \Pcal_{k} \times \Wcal_{l} \times \Scal_{m} \times \Ucal_{n})$ as test functions, and then integrate equations \eqref{fd1} - \eqref{fd6} with respect to time to obtain
\bse
\begin{alignat}{2}
	\int_0^{T_f}\{ (a_0 + a_r)(\partial_t T_i, S) - b_r(\partial_t p_k, S)
	+ \frac{a_r}{2\beta}(\partial_t \bsig_m, S\bid) + (\etab \cdot \bw_l, S) + (\nabla \cdot \br_j,S) \} \dt 
	&= \int_0^{T_f}(h,S)\dt, \label{fdd5}\\
	\int_0^{T_f}\{ (\bth^{-1} \br_j, \by) - (T_i, \nabla \cdot \by) \} \dt &= 0.\label{fdd6}\\
	\int_0^{T_f} \{(c_0 + c_r)(\partial_t p_k, q) - b_r(\partial_t T_i, q) 
	+ \frac{c_r}{2\alpha} (\partial_t \bsig_m, q\bid) + (\nabla \cdot \bw_l,q) \} \dt 
	&= \int_0^{T_f}(g,q)\dt,\label{fdd3}\\
	\int_0^{T_f}\{ (\bK^{-1} \bw_l, \bz) - (p_k,\nabla \cdot \bz) \} \dt &= 0, \label{fdd4}\\
	\int_0^{T_f}\{(\Acal \bsig_m, \btau) + (\bu_n, \nabla \cdot \btau) 
	+ \frac{c_r}{2\alpha}(\bid p_k, \btau) + \frac{a_r}{2\beta}(\bid T_i,\btau) \} \dt &= 0, \label{fdd1}\\
	-\int_0^{T_f}(\nabla \cdot \bsig_m, \bv) \dt &= \int_0^{T_f}(\bff,\bv)\dt. \label{fdd2}
\end{alignat}
\ese
Passing to the limit yields
\bse
\begin{alignat}{2}
	\int_0^{T_f}\{ (a_0 + a_r)(\partial_t T, S) - b_r(\partial_t p, S)
	+ \frac{a_r}{2\beta}(\partial_t \bsig, S\bid) + (\etab \cdot \bw, S) + (\nabla \cdot \br,S) \}\dt 
	&= \int_0^{T_f}(h,S)\dt, \label{var25}\\
	\int_0^{T_f}\{(\bth^{-1} \br, \by) - (T, \nabla \cdot \by) \} \dt &= 0. \label{var26}\\
	\int_0^{T_f}\{ (c_0 + c_r)(\partial_t p, q) - b_r(\partial_t T, q) 
	+ \frac{c_r}{2\alpha} (\partial_t \bsig, q\bid) + (\nabla \cdot \bw,q) \}\dt 
	&= \int_0^{T_f}(g,q)\dt,\label{var23}\\
	\int_0^{T_f}\{ (\bK^{-1} \bw, \bz) - (p,\nabla \cdot \bz) \} \dt &= 0, \label{var24}\\
	\int_0^{T_f}\{ (\Acal \bsig, \btau) + (\bu, \nabla \cdot \btau) 
	+ \frac{c_r}{2\alpha}(\bid p, \btau) + \frac{a_r}{2\beta}(\bid T,\btau)\} \dt &= 0, \label{var21}\\
	-\int_0^{T_f}(\nabla \cdot \bsig, \bv) \dt &= \int_0^{T_f}(\bff,\bv)\dt. \label{var22}
\end{alignat}
\ese
Finally, by the density of the test function space, $C^1(J;\Tcal_{i} \times \Rcal_{j} \times \Pcal_{k} \times \Wcal_{l} \times \Scal_{m} \times \Ucal_{n})$ in\\
 $L^2(J; \Tcal \times \Rcal \times \Pcal \times \Wcal \times \Scal \times \Ucal)$ as $(i, j, k, l, m, n) \rightarrow \infty$, the equations \eqref{linearized} hold true for a.e. $t \in J$. It remains now to show that the initial conditions are satisfied, i.e. $T(0) = T_0$, $\bu(0) = \bu_0$ and $p(0) = p_0$, in the weak sense. To this end, take $q \in C^1(J;\Pcal_k)$ such that $q(T_f) = 0$ as a test function in \eqref{fdd3} and integrate the first term by parts in time
\begin{equation}
	\begin{split}
	&\int_0^{T_f}\{-(c_0 + c_r)(p_k, \partial_t q) - b_r(\partial_t T_i, q) 
	+ \frac{c_r}{2\alpha} (\partial_t \bsig_m, q\bid) + (\nabla \cdot \bw_l,q)\}\dt \\
	&\qquad\qquad\qquad\qquad= \int_0^{T_f}(g,q)\dt + (c_0 + c_r)(p_k(0),q(0)).
	\end{split}
\end{equation}
On the other hand, from \eqref{var23} we obtain
\begin{equation}
	\begin{split}
	&\int_0^{T_f}\{-(c_0 + c_r)(p, \partial_t q) - b_r(\partial_t T, q) 
	+ \frac{c_r}{2\alpha} (\partial_t \bsig, q\bid) + (\nabla \cdot \bw,q)\}\dt \\
	&\qquad\qquad\qquad\qquad= \int_0^{T_f}(g,q)\dt + (c_0 + c_r)(p(0),q(0)).
	\end{split}
\end{equation}
Since $q(0)$ was arbitrary, and since $p_n(0) \rightarrow p_0$ in $L^2(\Omega)$, we get that $p(0) = p_0$. We obtain in the same way that $\bu(0) = \bu_0$, and $T(0) = T_0$. \\
$\bullet$ To finish the proof we show the uniqueness of a weak solution to problem \eqref{linearized}. To this end, assume that $(T_1(t), \br_1(t), p_1(t), \bw_1(t), \bsig_1(t), \bu_1(t))$ and  $(T_2(t), \br_2(t), p_2(t), \bw_2(t), \bsig_2(t), \bu_2(t))$ are two solution tuples in $\Tcal \times \Rcal \times \Pcal \times \Wcal \times \Scal \times \Ucal$, and let $(e_T(t), \be_{\br}(t), e_p(t), \be_{\bw}(t), \be_{\bsig}(t), \be_{\bu}(t))$ be the corresponding difference. This then satisfies the following variational problem: find $(e_T(t), \be_{\br}(t), e_p(t), \be_{\bw}(t), \be_{\bsig}(t), \be_{\bu}(t)) \in \Tcal \times \Rcal \times \Pcal \times \Wcal \times \Scal \times \Ucal$ such that for a.e. $t \in J$ there holds
\bse\label{evar}
	\begin{alignat}{2}
		(a_0 + a_r)(\partial_t e_T, S) - b_r(\partial_t e_p, S) + \frac{a_r}{2\beta}(\partial_t \be_{\bsig}, S\bid) 
		- (\etab \cdot \be_{\bw}, S) + (\nabla \cdot \be_{\br},S) &= 0, 
		&&\quad\forall S \in \Tcal,\label{e5}\\
		(\bth^{-1} \be_{\br}, \by) - (e_T, \nabla \cdot \by) &= 0, 
		&&\quad\forall \by \in \Rcal, \label{e6}\\
		(c_0 + c_r)(\partial_t e_p, q) - b_r(\partial_t e_T, q) + \frac{c_r}{2\alpha} (\partial_t \be_{\bsig}, q\bid) 
		+ (\nabla \cdot \be_{\bw},q) &= 0,
		&&\quad\forall q \in \Pcal, \label{e3}\\
		(\bK^{-1} \be_{\bw}, \bz) - (e_p,\nabla \cdot \bz) &= 0, 
		&&\quad\forall \bz \in \Wcal, \label{e4}\\
		(\Acal \be_{\bsig}, \btau) + (\be_{\bu}, \nabla \cdot \btau) + \frac{c_r}{2\alpha}(\bid e_p, \btau) 
		+ \frac{a_r}{2\beta}(\bid e_T,\btau) &= 0, 
		&&\quad\forall \btau \in \Scal, \label{e1}\\
		(\nabla \cdot \be_{\bsig}, \bv) &= 0,
		&&\quad\forall \bv \in \Ucal, \label{e2}
	\end{alignat}
\ese
together with homogeneous initial conditions.
Take now $\btau = \partial_t \be_{\bsig}$ in \eqref{e1}, differentiate \eqref{e2} with respect to time and set $\bv = \be_{\bu}$, $q = e_p$ in \eqref{e3}, $\bz = \be_{\bw}$ in \eqref{e4}, $S = e_T$ in \eqref{e5}, and $\by = \be_{\br}$ in \eqref{e6}, and add the resulting equations together
\begin{align}
		\nonumber&(c_0 + c_r)\frac{1}{2}\ddt (e_p,e_p) + (a_0 + a_r)\frac{1}{2} \ddt (e_T, e_T) 
		+ (\bK^{-1} \be_{\bw}, \be_{\bw}) + (\bth^{-1} \be_{\br}, \be_{\br}) \\
		&\hspace{5cm}= \frac{1}{2}\ddt(\Acal \be_{\bsig}, \be_{\bsig}) + b_r \ddt (e_p, e_T) - (\etab \cdot \be_{\bw}, e_T).
\end{align}
Integrating from $0$ to $t$ and using the properties of $\bK$ and $\bth$, in addition to the C-S and Young inequalities yields
\begin{align}\label{ehelp1}
		\nonumber&(c_0 + c_r)\frac{1}{2}\norm{e_p(t)}^2 + (a_0 + a_r)\frac{1}{2} \norm{e_T(t)}^2 
		+ \int_0^t \left(k_m \norm{\be_{\bw}(\tau)}^2 + \theta_m \norm{\be_{\br}(\tau)}^2\right)\dtau \\
		&\hspace{2cm}\leq \frac{1}{2} \norm{\be_{\bsig}(t)}^2_{\Acal} + \frac{b_r}{2} \norm{e_p(t)}^2 + \frac{b_r}{2}\norm{e_T(t)}^2
		+ \int_0^t \left(\gamma \frac{\epsilon}{2} \norm{\be_{\bw}(\tau)}^2 +\frac{1}{2\epsilon} \norm{e_T(\tau)}^2 \right)\dtau,
\end{align}
for some $\epsilon > 0$. 

On the other hand, from \eqref{e1} and \eqref{e2} we obtain
\begin{align}
	\nonumber \norm{\be_{\bsig}}^2_{\Acal} &= - \frac{c_r}{2\alpha} (\bid e_p, \be_{\bsig}) + \frac{a_r}{2\beta} (\bid e_T, \be_{\bsig}) \\
	&\leq \left( \frac{c_r}{2\alpha} \frac{\epsilon_1}{2} + \frac{a_r}{2\beta} \frac{\epsilon_2}{2} \right) 2(\mu + \lambda)
	\norm{\be_{\bsig}}^2_{\Acal} + \frac{c_r}{2\alpha} \frac{1}{2\epsilon_1} \norm{e_p}^2 
	+ \frac{a_r}{2\beta} \frac{1}{2\epsilon_2}\norm{e_T}^2.
\end{align}
Choosing $\epsilon_1 = \dfrac{1}{2\alpha}$ and $\epsilon_2 = \dfrac{1}{2\beta}$, we get
\begin{equation}\label{ehelp2}
	\frac{1}{2} \norm{\be_{\bsig}}^2_{\Acal} \leq \frac{c_r}{2} \norm{e_p} + \frac{a_r}{2} \norm{e_T}.
\end{equation}
Combining now \eqref{ehelp1} and \eqref{ehelp2}, and choosing $\epsilon = \dfrac{k_m}{\gamma}$, we get
\begin{equation}
	\begin{split}
		\frac{1}{2} \left( (c_0 - b_r) \norm{e_p(t)}^2 + (a_0 - b_r)\norm{e_T(t)}^2 \right) 
		+ \int_0^t \left(\frac{k_m}{2} \norm{\be_{\bw}(\tau)} + \theta_m \norm{\be_{\br}(\tau)}^2 \right)\dtau 
		\leq \frac{\gamma}{2k_m} \int_0^t \norm{e_T(\tau)}^2 \dtau,
	\end{split}
\end{equation}
which after application of the Gr\"onwall inequality yields
\begin{equation}
\begin{split}
	(c_0 - b_r)\norm{e_p(t)}^2
	+ (a_0 - b_r)\norm{e_T(t)}^2 
	+  \int_0^t\left(k_m \norm{\be_{\bw}(\tau)}^2 
	+ 2\theta_m \norm{\be_{\br}(\tau)}^2\right)\dtau \leq 0.
\end{split}
\end{equation}
Then, using Thomas' Lemma~\ref{thomas} we take $\btau = \tilde{\bsig}(\cdot,t) \in \Scal$ in \eqref{e1}, such that for $t \in J$, $-\nabla \cdot \tilde{\bsig}(t) = \be_{\bu}(t)$ in $\Omega$, with $\norm{\tilde{\bsig}(t)} \leq C \norm{\be_{\bu}(t)}$ for some constant $C > 0$. Thus, we obtain
\begin{align}
	\nonumber \norm{\be_{\bu}}^2 = -(\be_{\bu}, \nabla \cdot \tilde{\bsig})&= (\Acal \be_{\bsig}, \tilde{\bsig}) 
	+ \frac{c_r}{2\alpha}(\bid e_p, \tilde{\bsig}) + \frac{a_r}{2\beta}
	(\bid e_T, \tilde{\bsig}) \\
	&\leq \norm{\tilde{\bsig}} \left( \frac{1}{2\mu} \norm{\be_{\bsig}} 
	+ \frac{c_r}{2\alpha}\norm{e_p} + \frac{a_r}{2\beta}\norm{e_T} \right) \\
	\implies \norm{\be_{\bu}} &\leq C(\norm{\be_{\bsig}} + \norm{e_p} + \norm{e_T}),
\end{align}
where the constant $C>0$ depends on the coefficients, domain and spatial dimension. This implies that $e_T(t) = \be_{\br}(t) = e_p(t) = \be_{\bw}(t) = \be_{\bsig}(t) = \be_{\bu}(t) = 0$, in $\Omega$, for a.e. $t \in J$, implying the uniqueness of a weak solution to problem \eqref{linearized}. Finally, thanks to Lemma~\ref{improvedregularity}, we can finish the proof of the second part 
of Theorem~\ref{linWP} using similar arguments.
\qed
\section{Analysis of the non-linear problem}\label{sec4}
We now consider the analysis of the mixed variational formulation for the fully nonlinear 
problem~\eqref{nonlinvar}. The analysis uses the results derived previously for the linear case,
in addition to the Banach Fixed Point Theorem~\ref{CM} in order to obtain a local solution to \eqref{nonlinvar} in time. We then proceed to extend this local solution by small increments until a global solution is obtained for any finite final time (see e.g.~\cite{hunter1996nonlinear, van2004crystal} where similar techniques are used). Precisely, an iterative solution procedure is introduced based on linearizing the  heat flux term in~\eqref{nonlinenergy},  which is shown to be well-defined, and which converges to the weak solution of the nonlinear problem in adequate norms. Note that we now must require the iterates to be continuous in time, hence we shall invoke Lemma \ref{improvedregularity}. The iterative linearization algorithm we consider is then as follows: let $m \geq 1$, and at the iteration $m$, we 
solve for $(T^{m}, \br^{m}, p^{m}, \bw^{m}, \bsig^{m}, \bu^{m}) \in \Tcal \times \Rcal \times \Pcal \times \Wcal \times \Scal \times \Ucal$ such that for $t \in J$ there holds
\bse\label{linscheme}
	\begin{alignat}{2}
		\nonumber(a_0 + a_r) (\partial_t T^{m},S) - b_r(\partial_t p^{m},S) + \frac{a_r}{2\beta} (\partial_t \bsig^{m}, S \bid)\hspace{2cm}&&&\\ 
		+ (\nabla \cdot \br^{m}, S) + (\bw^{m} \cdot \bth^{-1} \br^{m-1}, S)&= (h,S),  
		&&\quad\forall S \in \Tcal, \label{iter5} \\
		(\bth^{-1} \br^{m}, \by) - (T^{m}, \nabla \cdot \by) &= 0,
		&&\quad\forall \by \in \Rcal, \label{iter6}\\
		(c_0 + c_r)(\partial_t p^{m}, q) - b_r(\partial_t T^{m}, q)+ \frac{c_r}{2\alpha} (\partial_t \bsig^{m}, q\bid) 
		+ (\nabla \cdot \bw^{m},q) &= (g,q),
		&&\quad\forall q \in \Pcal, \label{iter3}\\
		(\bK^{-1} \bw^{m}, \bz) - (p^{m},\nabla \cdot \bz) &= 0, 
		&&\quad\forall \bz \in \Wcal, \label{iter4}\\
		(\Acal \bsig^{m}, \btau) + (\bu^{m}, \nabla \cdot \btau) + \frac{c_r}{2\alpha}(\bid p^{m}, \btau) 
		+ \frac{a_r}{2\beta}(\bid T^{m},\btau) &= 0, 
		&&\quad\forall \btau \in \Scal,\label{iter1}\\
		-(\nabla \cdot \bsig^{m}, \bv) &= (\bff,\bv), 
		&&\quad\forall \bv \in \Ucal, \label{iter2}
	\end{alignat}
\ese
together with initial conditions, \eqref{weakICs}, and where the algorithm is initialized by given initial guess $\br^0$.
 We consider the following hypothesis on the heat flux:
\begin{Hyp}[The heat flux]\label{assumpdata2}
 We suppose that for all $m \geq 1$, the heat flux is such that $\br^m(t) \in L^\infty(\Omega)$, for $t\in J$.
\end{Hyp}
The above hypothesis is a natural one, and it is necessary for the solution to the iterative procedure \eqref{linscheme} to be well-defined for each $m \geq 1$. This hypothesis is satisfied with sufficiently regular data and domain boundary. We provide some formal arguments in Appendix~\ref{appendix} on the specific requirements such that the solution to the problem \eqref{linearized} yields $\br \in C([0,T_f],L^\infty(\Omega))$ (or alternatively $\bw, \br \in C([0,T_f];L^4(\Omega))$), thus making the above hypothesis superfluous. We delegate this discussion to the Appendix in order to avoid overly strict assumptions on the data. 
\begin{remark}
Note that if we we had approximated the convective term in equation \eqref{iter5} 
instead as $(\bw^{m-1} \cdot \bth^{-1} \br^{m}, S)$, Hypothesis~\ref{assumpdata2}
would be on the regularity of the Darcy flux $\bw$, and the above algorithm would be initialized by some $\bw^0$. The analysis presented next remains true and follows exactly the same lines.
\end{remark}
Based on the development of the previous sections, we now state the main result of this article.
\begin{theorem}\label{nonlinthm}
	Assume that $\bff$ is in $H^2(J;L^2(\Omega))$, $g,h$ in $H^1(J;L^2(\Omega))$, 
	$p_0, T_0$ in $H^1_0(\Omega), \textnormal{ and } \bu_0$ in $(L^2(\Omega))^d$,
	then the algorithm \eqref{linscheme}, initialized by any $\br^0 \in C([0,T_f];L^\infty(\Omega))$, 
	defines a unique sequence of iterates
\begin{subequations}
	\begin{align}
		&(T^{m}, \br^{m}) \in W^{1,\infty}(J;L^2(\Omega)) \times
		\left( L^\infty(J;H(\divr;\Omega)) \cap H^1(J;L^2(\Omega)) \right),\\
		&(p^{m}, \bw^{m}) \in W^{1,\infty}(J;L^2(\Omega)) \times
		\left( L^\infty(J;H(\divr;\Omega)) \cap H^1(J;L^2(\Omega)) \right),\\
		&(\bu^{m},\bsig^{m}) \in W^{1,\infty}(J;L^2(\Omega)) 
		\times\left(  L^\infty(J;H_s(\divr;\Omega))\cap W^{1,\infty}(J;L^2(\Omega))\right), 
	\end{align}
\end{subequations}
that converges to the weak solution $(T, \br, p, \bw, \bsig, \bu)$ of \eqref{nonlinvar}, admitting the following regularity
\begin{subequations}
	\begin{align}
		&(T, \br) \in H^1(J;L^2(\Omega)) \times \left(L^2(J;H(\divr;\Omega)) \cap L^\infty(J;L^2(\Omega)) \right),\\
		&(p, \bw) \in H^1(J;L^2(\Omega)) \times \left( L^2(J;H(\divr;\Omega)) \cap L^\infty(J;L^2(\Omega)) \right),\\
		&(\bu,\bsig) \in H^1(J;L^2(\Omega)) 
		\times \left(L^2(J;H_s(\divr;\Omega))\cap H^1(J;L^2(\Omega))\right).
	\end{align}
\end{subequations} 
\end{theorem}
\begin{proof}
According to Theorem \ref{linWP} and regarding Hypothesis~\ref{assumpdata2}, 
the iterates $(T^{m}, \br^{m}, p^{m}, \bw^{m}, \bsig^{m}, \bu^{m})$ are well-defined for all $m\geq1$, admitting the improved regularity specified in Lemma \ref{improvedregularity}. In particular, this guarantees continuity in time for the iterates. Keeping this in mind, we define $\gamma_1 := \sup_{t \in J} \norm{\be_{\bw}^{m}(t)}^2$ and $\gamma_1 := \sup_{t \in J} \norm{\be_{\br}^{m}(t)}^2$. It remains to show the convergence of the iterates to the weak solution of~\eqref{nonlinvar} using suitable norms. To this aim, let $m\geq2$, and take the difference of equations \eqref{linscheme} at the iteration step $m$, with the corresponding equations at iteration step $m-1$ to obtain the following problem: find $(e_T^{m}, \be_{\br}^{m}, e_p^{m}, \be_{\bw}^{m}, \be_{\bsig}^{m}, \be_{\bu}^{m}) \in \Tcal \times \Rcal \times \Pcal \times \Wcal \times \Scal \times \Ucal$ such that for $t \in J$ there holds
\bse\label{erroreqs}
\begin{alignat}{2}
	\nonumber(a_0 + a_r)(\partial_t e_T^{m}, S) - b_r(\partial_t e_p^{m}, S) + \frac{a_r}{2\beta} (\partial_t \be_{\bsig}^{m}, S \bid) + (\nabla \cdot \be_{\br}^{m}, S)\\  
	\qquad \qquad \qquad \qquad \qquad \qquad- (\bw^{m} \cdot \bth^{-1} \be_{\br}^{m-1}, S) 
	- (\be_{\bw}^{m} \cdot \bth^{-1} \br^{m-1}, S)  &= 0,&&\quad\forall S \in \Tcal, \label{itere5} \\
	(\bth^{-1} \be_{\br}^{m}, \by) - (e_T^{m}, \nabla \cdot \by) &= 0,&&\quad\forall \by \in \Rcal \label{itere6}\\
	(c_0 + c_r)(\partial_t e_p^{m}, q) - b_r(\partial_t e_T^{m}, q) + \frac{c_r}{2\alpha} (\partial_t \be_{\bsig}^{m}, q\bid) 
	+ (\nabla \cdot \be_{\bw}^{m}, q) &= 0,&&\quad \forall q \in \Pcal,\label{itere3}\\
	(\bK^{-1} \be_{\bw}^{m}, \bz) - (e_p^{m}, \nabla \cdot \bz) &= 0,&&\quad \forall \bz \in \Wcal, \label{itere4} \\
	(\Acal \be_{\bsig}^{m}, \btau) + (\be_{\bu}^{m}, \nabla \cdot \btau) + \frac{c_r}{2\alpha} (\bid e_p^{m}, \btau) 
	+ \frac{a_r}{2\beta} (\bid e_T^{m} , \btau) &= 0,&&\quad \forall \btau \in \Scal, \label{itere1}\\
	-(\nabla \cdot \be_{\bsig}^{m}, \bv) &= 0,&&\quad \forall \bv \in \Ucal, \label{itere2}
\end{alignat}
together with homogeneous initial conditions, i.e., 
\begin{equation}\label{weakICserror}
(e_T^{m}(0),S) =0, \quad \forall  \in \Tcal, \quad ( \be^{m}_{\bu}(0),\bv) = 0, \quad \forall \bv \in \Ucal, \quad \textnormal{ and } \quad
(e^{m}_{p}(0),q) = 0, \quad \forall q \in \Pcal.
\end{equation}
\ese
The solution tuple 
$(e_T^{m}, \be_{\br}^{m}, e_p^{m}, \be_{\bw}^{m}, \be_{\bsig}^{m}, \be_{\bu}^{m})$ denotes 
the error functions between the solution to~\eqref{linscheme} at the ${m}^{\textnormal{th}}$ and $(m-1)^{\textnormal{th}}$ iterations, i.e. $e_T^{m} = T^{m} - T^{m-1}$, and similarly for the other variables. First, take $\btau = \be_{\bsig}^{m}$ and $\bv = \be_{\bu}^{m}$ in equations \eqref{itere1} and \eqref{itere2}, respectively, and sum to obtain
\begin{equation}
\begin{split}
\norm{\be_{\bsig}^{m}}^2_{\Acal} &= - \frac{c_r}{2\alpha} (\bid e_p^{m}, \be_{\bsig}^{m}) - \frac{a_r}{2\beta} (\bid e_T^{m}, \be_{\bsig}^{m}) \\
&\leq \left( \alpha \frac{\epsilon_1}{2} + \beta \frac{\epsilon_2}{2} \right) \norm{\be_{\bsig}^{m}}^2_{\Acal} 
+ \frac{c_r}{2\alpha} \frac{1}{2\epsilon_1} \norm{e_p^{m}}^2 + \frac{a_r}{2\beta} \frac{1}{2\epsilon_1} \norm{e_T^{m}}^2.
\end{split}
\end{equation}
Setting $\epsilon_1 = \dfrac{1}{2\alpha}$ and $\epsilon_2 = \dfrac{1}{2\beta}$ yields
\begin{equation}\label{esigbound}
\norm{\be_{\bsig}^{m}}^2_{\Acal} \leq c_r\norm{e_p^{m}}^2 + a_r\norm{e_T^{m}}^2.
\end{equation}
Similarly, by differentiating equations \eqref{itere1} and \eqref{itere2} with respect to time and setting $\btau = \pt \be_{\bsig}^{m}$ and $\bv = \pt \be_{\bu}^{m}$ we obtain
\begin{equation}\label{ptesigbound}
\norm{\pt \be_{\bsig}^{m}}^2_{\Acal} \leq c_r\norm{\pt e_p^{m}}^2 + a_r\norm{\pt e_T^{m}}^2.
\end{equation}
Using Thomas' lemma~\ref{thomas}, we take $\btau = \tilde{\bsig}(\cdot, t)$ in equation \eqref{itere1} such that $\be_{\bu}^{m}(\cdot, t) = \nabla \cdot \tilde{\bsig}(\cdot, t)$ with $\norm{\tilde{\bsig}(t)} \leq C \norm{\be_{\bu}^{m}(t)}$ for $t \in J$, and combine with \eqref{esigbound} to obtain
\begin{equation}\label{eubound}
	\norm{\be_{\bu}^{m}}^2 \leq C\left(\norm{e_p^{m}}^2 + \norm{e_T^{m}}^2\right),
\end{equation}
and similarly using \eqref{ptesigbound}
\begin{equation}\label{pteubound}
	\norm{\partial_t \be_{\bu}^{m}}^2 \leq C\left(\norm{\partial_t e_p^{m}}^2 
	+ \norm{\partial_t e_T^{m}}^2\right),
\end{equation}
where the constants $C>0$ are independent of $m$. Now, write $\nabla \cdot \be_{\br}^{m}(t) = \sum_{\ell=1}^\infty \zeta_\ell(t) S_\ell$ for some functions $\zeta_\ell(t) \in \real$, where $\spann \{S_\ell : 1\leq \ell \leq \infty\} = \Tcal$. Then, we take $S_\ell$ as a test function in equation \eqref{itere5}, multiply by $\zeta_\ell$ and sum over $\ell = 1,...,k$ to obtain
\begin{align}
	\nonumber(\nabla \cdot \be_{\br}^{m}, \sum_{\ell=1}^k \zeta_\ell S_\ell) &=
	 b_r (\partial_t e_p^{m}, \sum_{\ell=1}^k \zeta_\ell S_\ell) 
	 - (a_0 + a_r) (\partial_t e_T^{m}, \sum_{\ell=1}^k \zeta_\ell S_\ell)
	 - \frac{a_r}{2\beta} (\partial_t \be_{\bsig}^{m}, \sum_{\ell=1}^k \zeta_\ell S_\ell) \\
	&\hspace{1cm}+ (\bw^{m} \cdot \bth^{-1} \be_{\br}^{m-1}, \sum_{\ell=1}^k \zeta_\ell S_\ell) 
	+ (\be_{\bw}^{m} \cdot \bth^{-1} \br^{m-1}, \sum_{\ell=1}^k \zeta_\ell S_\ell). 
\end{align}
Using the C-S and Young inequalities, tending $k \rightarrow \infty$, and using also the estimate \eqref{ptesigbound} we get
\begin{equation}\label{edivrbound}
\begin{split}
	\norm{\nabla \cdot \be_{\br}^{m}}^2 \leq &C\left(\norm{\partial_t e_p^{m}}^2 + \norm{\partial_t e_T^{m}}^2 
	+ \norm{\be_{\bw}^{m}}^2 + \norm{\be_{\br}^{m-1}}^2\right).
\end{split}
\end{equation}
In the same way we get from equation \eqref{itere3} that
\begin{equation}\label{edivwbound}
\begin{split}
	\norm{\nabla \cdot \be_{\bw}^{m}}^2 \leq &C\left(\norm{\partial_t e_T^{m}}^2 + \norm{\partial_t e_p^{m}}^2 \right),
\end{split}
\end{equation}
where the constants $C>0$ depends on $\gamma_1$ and $\gamma_2$ but is independent of $m$. From \eqref{itere2} we also have that
\begin{equation}\label{edivsigbound}
\norm{\nabla \cdot \be_{\bsig}^{m}}^2 = 0.
\end{equation}
We continue by setting $S = e_T^{m}, \by = \be_{\br}^{m}, q = e_p^{m}, \bz = \be_{\bw}^{m}, \btau = \pt \be_{\bsig}^{m}$ in equations \eqref{itere5}--\eqref{itere1}, and differentiate equation \eqref{itere2} with respect to time and set $\bv = \be_{\bu}^{m}$. Summing the resulting equations yields
\begin{equation}
\begin{split}
(c_0 - b_r) \frac{1}{2} \ddt \norm{e_p^{m}}^2 &+ (a_0 - b_r) \frac{1}{2} \ddt \norm{e_T^{m}}^2 
+ k_m \norm{\be_{\bw}^{m}}^2 + \theta_m \norm{\be_{\br}^{m}}^2 \\
&\qquad \leq \gamma_1 \frac{\theta_M}{2} \norm{\be_{\br}^{m-1}}^2 + \gamma_2 \theta_M\frac{\epsilon}{2} \norm{\be_{\bw}^{m}}^2
+ \left(\frac{1}{2} + \frac{1}{2\epsilon} \right)\norm{e_T^{m}}^2,
\end{split}
\end{equation}
where we also used the estimate \eqref{ptesigbound}. Integrating from $0$ to $t$, applying the Gr\"onwall inequality and setting $\epsilon = \dfrac{k_m}{\gamma_2 \theta_M}$ yields
\begin{equation}\label{epetbound}
\begin{split}
(c_0 - b_r) \norm{e_p^{m}(t)}^2 + (a_0 - b_r) \norm{e_T^{m}(t)}^2 
+ &\int_0^t \left( k_m\norm{\be_{\bw}^{m}(\tau)}^2 + \theta_m\norm{\be_{\br}^{m}(\tau)}^2\right)\dtau\\
&\hspace{4cm}\leq C\int_0^t \norm{\be_{\br}^{m-1}(\tau)}^2\dtau,
\end{split}
\end{equation}
for some constant $C>0$ independent of $m$. Take now $S = \partial_t e_T^{m}$ and $q = \partial_t e_p^{m}$ in equations \eqref{itere5} and \eqref{itere3}, respectively. Then, differentiate equations \eqref{itere1} and \eqref{itere2} with respect to time and let $\btau = \partial_t \bsig^{m}$ and $\bv = \partial_t \bu^{m}$. Finally, we let $\by = \partial_t \be_{\br}^{m}$ and $\bz = \partial_t \be_{\bw}^{m}$ in equations \eqref{itere6} and \eqref{itere4}, respectively. Summing yields
\begin{equation}
\begin{split}
&(c_0 + c_r - b_r) \norm{\pt e_p^{m}}^2 + (a_0 + a_r - b_r) \norm{\pt e_T^{m}}^2 
+ \frac{k_m}{2} \ddt \norm{\be_{\bw}^{m}}^2 + \frac{\theta_m}{2} \ddt \norm{\be_{\br}^{m}}^2\\
&\qquad \leq \norm{\pt \be_{\bsig}^{m}}^2_{\Acal} + (\bw^{m} \cdot \bth^{-1} \be_{\br}^{m-1}, \pt e_T^{m}) + (\be_{\bw}^{m} \cdot \bth^{-1} \br^{m-1}, \pt e_T^{m})\\
&\qquad \leq \norm{\pt \be_{\bsig}^{m}}^2_{\Acal} + \left(\frac{\epsilon_1}{2} + \frac{\epsilon_2}{2} \right) \norm{\pt e_T^{m}}^2
+ \gamma_1 \theta_M \frac{1}{2\epsilon_1} \norm{\be_{\bw}^{m}}^2 + \gamma_2 \theta_M \frac{1}{2\epsilon_2} \norm{\be_{\br}^{m-1}}^2,
\end{split}
\end{equation}
for some $\epsilon_1, \epsilon_2 > 0$. Combining this with the previous estimate \eqref{ptesigbound} and setting $\epsilon_1 = \epsilon_2 = \dfrac{\alpha \beta}{\mu + \lambda}$ leads to
\begin{align}\label{contr}
\nonumber&(c_0 - b_0) \norm{\pt e_p^{m}}^2 + (a_0 - b_0) \norm{\pt e_T^{m}}^2 
+ \frac{k_m}{2} \ddt \norm{\be_{\bw}^{m}}^2 + \frac{\theta_m}{2} \ddt \norm{\be_{\br}^{m}}^2\\
&\hspace{6cm}\leq  \frac{\theta_M}{2} \frac{\mu + \lambda}{\alpha \beta} 
\left(\gamma_1 \norm{\be_{\bw}^{m}}^2 + \gamma_2 \norm{\be_{\br}^{m-1}}^2\right).
\end{align}
Integrating from $0$ to $t$ and applying the Gr\"onwall inequality yields
\begin{align}\label{ptet}
\nonumber&(c_0 - b_0)\int_0^t \norm{\pt e_p^{m}(\tau)}^2\dtau + (a_0 - b_0) \int_0^t\norm{\pt e_T^{m}(\tau)}^2 \dtau
+ \frac{k_m}{2} \norm{\be_{\bw}^{m}(t)}^2 + \frac{\theta_m}{2}  \norm{\be_{\br}^{m}(t)}^2\\
&\hspace{3cm}\leq  \frac{\xi \gamma_2}{2} \exp \left(\dfrac{\xi \gamma_1}{k_m} T_f\right)  \int_0^t\norm{\be_{\br}^{m-1}(\tau)}^2\dtau 
\leq   \frac{\xi \gamma_2}{2} \exp\left(\dfrac{\xi \gamma_1}{k_m} T_f\right)  \int_0^{t_1}\norm{\be_{\br}^{m-1}(\tau)}^2\dtau .
\end{align}
for $t \leq t_1$ where $t_{1} > 0$ will be fixed later, and where $\xi = \theta_M \dfrac{\mu + \lambda}{\alpha \beta}$. Integrating in time once more from $0$ to $t_{1}$ yields
\begin{equation}
\int_0^{t_1} \norm{\be_{\br}^{m}(\tau)}^2\dtau 
\leq t_1 C_{\textnormal{contr}}\int_0^{t_1}\norm{\be_{\br}^{m-1}(\tau)}^2\dtau,
\end{equation}
where the constant $C_{\textnormal{contr}} = \dfrac{\xi \gamma_2}{2} \exp\left(\dfrac{\xi \gamma_1}{k_m} T_f\right)$ is such that $ 0 < C_{\textnormal{contr}} < \infty$ provided $T_f < \infty$, and is independent of $m$ and of  the local final time $t_{1}$. 
Thus, for $t_1= \dfrac{1}{2C_{\textnormal{contr}}}$ the above expression implies that the map $\be_{\br}^{m-1}(t) \mapsto \be_{\bw}^{m}(t)$ is a contraction map for $t \in (0,t_1]$. In particular, this implies that as $m \rightarrow \infty$ we have from Theorem \ref{CM} and \eqref{esigbound}--\eqref{pteubound}, \eqref{edivrbound}--\eqref{edivsigbound}, \eqref{epetbound} and \eqref{ptet} the following convergences
\begin{itemize} 
\item
$\be_{\bw}^{m}, \be_{\br}^{m} \rightarrow 0$ in $L^2(0,t_1; H(\divr,\Omega)) \cap L^\infty(0,t_1; L^2(\Omega))$,
\item
$e_p^{m}, e_T^{m} \rightarrow 0$ in $H^1(0,t_1;L^2(\Omega))$,
\item
$\be_{\bsig}^{m} \rightarrow 0$ in $H^1(0,t_1;L^2(\Omega)) \cap L^2(0,t_1;H_s(\divr,\Omega))$,
\item
$\be_{\bu}^{m} \rightarrow 0$ in $H^1(0,t_1; L^2(\Omega))$.
\end{itemize}
Therefore, the existence of the solution to problem \eqref{nonlinvar} is established for $t \in (0,t_1]$. 
The question now is how  
to continue the local solution $(T, \br, p, \bw, \bsig, \bu)$ to the 
system~\eqref{nonlinvar} globally in time. To this aim, we let 
$(T^{m}, \br^{m}, p^{m}, \bw^{m}, \bsig^{m}, \bu^{m})$ be the solution 
of~\eqref{linscheme} on the time interval $[t_{k-1},t_{k}]$, $k\in\nat$, with  $t_{k}-t_{k-1}=\dfrac{1}{2C_{\textnormal{contr}}}$, and 
starting with the initial 
data $(T^{m}, \br^{m}, p^{m}, \bw^{m}, \bsig^{m}, \bu^{m})(\cdot,t_{k-1})=(T, \br, p, \bw, \bsig, \bu)|_{[t_{k-2},t_{k-1}]}(\cdot,t_{k-1})$; 
thanks to the continuity in-time of the convergent solution. 
The iterates $(T^{m}, \br^{m}, p^{m}, \bw^{m}, \bsig^{m}, \bu^{m})$ are again 
well-defined using Theorem \ref{linWP} and  Hypothesis~\ref{assumpdata2}. The iterates also 
result  a contraction, i.e,
\begin{equation}
\int_{t_{k-1}}^{t_{k}} \norm{\be_{\br}^{m}(\tau)}^2\dtau 
\leq \dfrac{1}{2}\int_{t_{k-1}}^{t_{k}}\norm{\be_{\br}^{m-1}(\tau)}^2\dtau,\quad \forall k\geq 2.
\end{equation}
Therefrom, we proceed as on done in the first time interval $[0,t_{1}]$
to show  the convergence of the successive approximations  
$(T^{m}, \br^{m}, p^{m}, \bw^{m}, \bsig^{m}, \bu^{m})|_{[t_{k-1},t_{k}]}$, $k\in\nat$, to 
$(T, \br, p, \bw, \bsig, \bu)|_{[t_{k-1},t_{k}]}$. This solution is 
similarly   extended to any  time $t_{\ell}\geq t_{k}$ given by 
$$t_{\ell}=\sum_{k=1}^{\ell}t_{k}-t_{k-1} = \sum_{k=1}^{\ell}\frac{1}{2C_{\textnormal{contr}}}.$$
Finally, since the  series 
$\sum_{k=1}^{\infty}\frac{1}{2C_{\textnormal{contr}}}$ diverges, the sequence
of local solutions is extended to arbitrary finite final time $0 <T_f < \infty$ by incrementing the values $\ell$ (if $T_f$ is not identically an integer multiple of $\frac{1}{2C_{\textnormal{contr}}}$ take instead $t_k - t_{k-1} = \frac{1}{NC_{\textnormal{contr}}}$ where $N>1$). This concludes the proof of Theorem \ref{nonlinthm}.\end{proof}
\begin{remark}
The choice of the iterative procedure defined in \eqref{linscheme} is not the only possible choice. For instance, we could also define a fully explicit scheme where both the Darcy and heat fluxes in the convective term are given at the previous iteration. If such an explicit scheme was chosen we would have the advantage of a symmetric linearized problem, as the convective terms in the iterative procedure can be viewed as part of the source term on the right hand side. However, it is well-known that in a practical setting explicit schemes has slower convergence rate than implicit schemes, and therefore we chose the latter.
\end{remark}
\begin{remark}
Assume that $\bff$ is in $H^1(J;L^2(\Omega))$, $g,h$ in $L^{^2}(J;L^2(\Omega))$, 
$p_0, T_0$ in $H^1_0(\Omega), \textnormal{ and } \bu_0$ in $(L^2(\Omega))^d$.  Suppose 
that instead of Hypothesis\ref{assumpdata2}, we have $\br^m$, $\bw^m$  in $\in H^1(0,T;L^{\infty}(\Omega))$.
Then, we can reproduce the proof of Theorem~\ref{nonlinthm} to prove the convergence 
of the scheme given  by~\eqref{linscheme} to a weak solution of the nonlinear 
problem~\eqref{nonlinvar}.
\end{remark}
\section{Conclusions}\label{concl}
In this article we have given mixed formulations
for  the   fully coupled  quasi-static thermo-poroelastic model~\cite{brun2018thporo}. The model in nonlinear, with the nonlinearity appearing on a coupling term. This makes the analysis very challenging. A linearization of the model was therefore employed as an intermediate step in analyzing the full nonlinear model. For the linear case, the well-posedness 
is established using the theory if DAEs, energy estimates, and Gr\"onwall's lemma,
together with a Galerkin method. This result together with 
derived energy estimates are combined with the Banach Fixed Point Theorem to obtain local solutions in time of the nonlinear problem. Due to the continuity in time of the convergent (local) solutions, we can infer a (global) convergence proof of an iterative procedure, designed to produce the weak solution 
to the original nonlinear problem. Work underway addresses discretization 
of this model problem using an appropriate  mixed finite element method as well as 
\textit{a priori} and \textit{a posteriori} 
error analysis, the same way as in~\cite{ahmed:hal-01687026}.
\appendix
\section{Alternative to Hypothesis 1}\label{appendix}
We outline some formal calculations which reveal the assumptions necessary on the data in order to avoid the Hypothesis \ref{assumpdata2}. In particular, we aim to solve the linear problem \ref{linearized} with sufficiently regular data such that $\br \in C([0,T_f];L^\infty(\Omega))$ (or, alternatively such that $\bw, \br \in C([0,T_f];L^4(\Omega))$. The following arguments indicate that this is easily done. First, note that from Theorem \ref{linWP} and the Sobolev Embedding Theorem it follows that the functions $(T(t), \br(t), p(t), \bw(t), \bsig(t), \bu(t))$ are continuous for $t \in [0,T_f]$, if $g,h \in H^1(J;L^2(\Omega))$ and $\bff \in H^2(J;L^2(\Omega))$. Thus, going back to the problem \eqref{linearized}, we can choose smooth test functions with compact support in $\Omega$ and find that $(T, \br, p, \bw, \bsig, \bu)$ solves the following initial boundary value problem
\bse
 	\begin{alignat}{2}
 	a_0 \tderiv{T}(t) - b_0 \tderiv{p}(t) + \frac{a_r}{2\beta} \tderiv{\tr{\bsig}}(t) - \etab \cdot \bw(t) + \nabla \cdot \br(t) &= h(t),\quad
 	&&\textnormal{in } \Omega, \label{A:heat}\\
 	\bth^{-1} \br(t) + \nabla T(t) &= 0,\quad 
 	&&\textnormal{in } \Omega, \label{A:heatflux}\\
 	c_0 \tderiv{p}(t) - b_0 \tderiv{T}(t) + \frac{c_r}{2\alpha} \tderiv{\tr{\bsig}}(t) + \nabla \cdot \bw(t) &= g(t),\quad
 	&&\textnormal{in } \Omega, \label{A:pressure}\\
 	\bK^{-1} \bw(t) + \nabla p(t) &= 0,\quad
 	&&\textnormal{in } \Omega,\label{A:darcyflux} \\
 	\Acal \bsig(t) - \bep(\bu)(t) + \frac{c_r}{2\alpha}\bid p(t) + \frac{a_r}{2\beta}\bid T(t)&= 0,\quad
 	&&\textnormal{in } \Omega,  \\
 	-\nabla \cdot \bsig(t) &= \bff(t),\quad &&\textnormal{in } \Omega, 		
 	\end{alignat}
for a.e. $t \in J$, and with boundary conditions
\begin{equation}
T = 0, \quad \bu = 0, \quad p = 0, \quad \textnormal{on } \Gamma \times J,
\end{equation}
and initial conditions
\begin{equation}
T(0) = T_0, \quad \bu(0) = \bu_0, \quad \textnormal{and } \quad p(0) = p_0, \quad \textnormal{in } \Omega \times \{0\}.
\end{equation}
\ese
Since $\bth^{-1}\br, \bK^{-1}\bw \in L^2(J;L^2(\Omega))$ we have from \eqref{A:heatflux} and \eqref{A:darcyflux} that $T, p \in L^2(J;H_0^1(\Omega))$. Thus, we can write \eqref{A:heat} and \eqref{A:pressure} in non-mixed form, i.e.
\begin{align}
a_0 \tderiv{T}(t) - b_0 \tderiv{p}(t) + \frac{a_r}{2\beta} \tderiv{\tr{\bsig}}(t) - \etab \cdot \bw(t) - \nabla \cdot (\bth \nabla T(t)) &= h(t), \\
c_0 \tderiv{p}(t) - b_0 \tderiv{T}(t) + \frac{c_r}{2\alpha} \tderiv{\tr{\bsig}}(t) - \nabla \cdot (\bK \nabla p(t)) &= g(t),
\end{align}
and use the theory of linear parabolic equations (see \cite{evans1998partial} p. 349 for details) to get increased regularity for $T(t)$ and $p(t)$, and then use \eqref{A:heatflux} and \eqref{A:darcyflux} to infer increased regularity for $\br(t)$ and $\bw(t)$. In particular, if the domain boundary $\Gamma$ is of class $C^1$, $h,g \in C^1([0,T_f];H^1(\Omega))$, $\bff \in C^2([0,T_f];L^2(\Omega))$ and $T_0 \in H_0^1(\Omega)\cap H^2(\Omega)$, then $T \in H^1(J;H^2(\Omega))$ and thus $\br \in H^1(J;H^1(\Omega))$. Due to the special case of the Sobolev embedding theorem for $d=2$, i.e. $H^1(\Omega) \subset L^\infty(\Omega)$, we get that $\br \in C([0,T_f];L^\infty(\Omega))$. Alternatively, if also $p_0 \in H_0^1(\Omega)\cap H^2(\Omega)$, then we have additionally $\bw \in H^1(J;H^1(\Omega))$, and since $H^1(\Omega) \subset L^4(\Omega)$ (independently of spatial dimension), we get $\br, \bw \in C([0,T_f];L^4(\Omega))$.

\section{Tables}\label{appB}
For easy reference we list some of the notations used in this article.
\begin{table}[h!]
\centering
\begin{tabular}{ ll } 
 \hline
 Parameter & Description \\ 
 \hline \hline 
 $a_0$ & effective thermal capacity \\
 $b_0$ & thermal dilation coefficient \\
 $\beta$ & thermal stress coefficient \\
 $\bK$ & matrix permeability divided by fluid viscosity \\
 $\bth$ & effective thermal conductivity \\
 $\mu, \lambda$ & Lam\'e parameters \\
 $\alpha$ & Biot-Willis constant \\
 $c_0$ & specific storage coefficient \\
 \hline
\end{tabular}
\caption{Description of parameters}
\end{table}

\begin{table}[h!]
\centering
\begin{tabular}{lll} 
\hline
Variable & Description & Spaces\\
\hline \hline 
$T$ & temperature distribution & $\Tcal := L^2(\Omega)$ \\
$\bu$ & solid displacement & $\Ucal := (L^2(\Omega))^d$\\
$p$ & fluid pressure & $\Pcal := L^2(\Omega)$ \\
$\bsig$ & total stress & $\Scal := H_s(\divr;\Omega)$\\
$\bw$ & Darcy flux &$\Wcal := H(\divr;\Omega)$ \\
$\br$ & heat flux & $\Rcal := H(\divr;\Omega)$ \\
\hline
\end{tabular}
\caption{Variables}
\end{table}



\section*{Acknowledgements}
This work forms part of Norwegian Research Council project 250223. The authors would also like to thank Kundan Kumar for very helpful discussions concerning the analysis of the nonlinear model.

\vspace*{-0.2cm}

\enlargethispage{0.2cm}

\bibliographystyle{siamplain}
\bibliography{ref}

\end{document}